\documentclass[a4paper]{article}
\usepackage[utf8]{inputenc}
\usepackage[a4paper,margin=1.4in]{geometry}
\usepackage{amsmath, amsthm, amsfonts, amssymb, slashed, stmaryrd, mathrsfs}
\usepackage{wrapfig}
\usepackage[dvipsnames]{xcolor}
\usepackage{tikz-cd}
\usepackage{aliascnt}
\usepackage{slashed}
\usepackage{graphicx}
\usepackage{leftindex}
\usepackage{svg}
\usepackage{float} 
\usepackage{mathabx}
\graphicspath{ {./figures/} }

\usepackage[
backend=biber,
style=alphabetic,
firstinits=true,
isbn=false,
url=false,
sorting=nyt,
maxnames=99,
minnames=99
]{biblatex}
\AtBeginBibliography{\small}

\usepackage[hidelinks]{hyperref} 
\hypersetup{
    colorlinks=true,
    linkcolor=MidnightBlue,
    filecolor=MidnightBlue,      
    urlcolor=MidnightBlue,
    citecolor=MidnightBlue,
    pdfpagemode=FullScreen,
    }

\newcommand{\rd}{{\mathrm d}}

\newcommand{\cC}{\mathcal{C}}

\newcommand{\sI}{\mathscr{I}}

\newcommand{\sO}{\mathscr{O}}

\newcommand{\dF}{{\rd F}}

\newcommand{\N}{\mathbb{N}}
\newcommand{\Z}{\mathbb{Z}}
\newcommand{\R}{\mathbb{R}}
\newcommand{\C}{\mathbb{C}}

\renewcommand{\Re}{\mathrm{Re}}

\newcommand{\Hom}{\mathrm{Hom}}

\numberwithin{equation}{section}

\theoremstyle{plain}

\newcommand{\mynewtheorem}[2]{
  \newaliascnt{#1}{dummy}
  \newtheorem{#1}[#1]{#2}
  \aliascntresetthe{#1}
  \expandafter\def\csname #1autorefname\endcsname{#2}
}

\theoremstyle{plain}
\mynewtheorem{theorem}{Theorem}
\mynewtheorem{lemma}{Lemma}
\mynewtheorem{proposition}{Proposition}
\mynewtheorem{corollary}{Corollary}
\mynewtheorem{conjecture}{Conjecture}

\theoremstyle{definition}
\mynewtheorem{example}{Example} 
\mynewtheorem{definition}{Definition}
\mynewtheorem{remark}{Remark}
\mynewtheorem{question}{Question}
\mynewtheorem{problem}{Problem}
\mynewtheorem{assumption}{Assumption}

\theoremstyle{plain}

\usepackage[english]{babel}
\addto\extrasenglish{

}

\newcommand{\reg}{\mathrm{reg}}

\title{A Morse complex for the homology of vanishing cycles}
\author{Aleksander Doan and Juan Muñoz-Echániz}
\date{} 

\begin{document}

\maketitle

\abstract{\noindent We construct Morse homology groups associated with any regular function on a smooth complex algebraic variety, allowing singular and non-compact critical loci. These groups are generated by critical points of a certain large pertubation of the function, built from a normal crossing compactification of the variety. They are canonically isomorphic to the homology of vanishing cycles and---in the absence of bifurcations at infinity---recover the hypercohomology of the perverse sheaf of vanishing cycles, studied extensively in singularity theory and enumerative geometry. Our construction arises as a special case of a more general construction of Morse homology of non-compact manifolds that admit a compactification by a manifold with corners. 

\setcounter{tocdepth}{1}  
\renewcommand{\contentsname}{} 
\tableofcontents
~\\

\section{Introduction}


The singular homology of a compact smooth manifold can be described in terms of the critical points and gradient flowlines of a Morse function, via the construction of \textit{Morse homology} \cite{thom-partition,smale-morseineq,milnor-hcobordism,witten-susy-morse-theory}. Morse homology serves as a finite-dimensional model for \textit{Floer homology} \cite{Floer1988a,Floer1988,floer-infinite-dimensional-morse-theory}, a powerful tool in low-dimensional topology and symplectic geometry. 

Motivated by the emerging framework of \textit{holomorphic Floer theories} \cite{Donaldson1998,Haydys2015,abouzaidmanolescu,doan-rezchikov,kontsevich-soibelman}, we provide a finite-dimensional model for such theories by constructing Morse homology for any regular function on a smooth complex algebraic variety, and relating it to known algebro-geometric invariants; see \autoref{theorem:main1}.  Our construction arises as a special case of a more general framework for Morse homology of non-compact manifolds, which is of independent interest in topology; see \autoref{theorem:main_real}.

\subsection{Morse homology of non-compact manifolds}
In the classical framework of Morse theory, $M$ is a compact smooth manifold and $f : M \to \mathbb{R}$ is a smooth function. After perturbing $f$ to satisfy the Morse condition and choosing a generic Riemannian metric $g$, one defines Morse homology $HM_*(M,f)$ as the homology of a chain complex generated by the critical points of $f$, with differential counting gradient trajectories of $f$ with respect to $g$. It is canonically isomorphic to the singular homology $H_*(M)$---in particular, independent of $f$ and $g$; see, e.g. \cite{schwarz-book}. 

More generally, when $M$ is non-compact, the Morse homology of $(M,f)$ can be defined provided that the critical set $\mathrm{Crit}(f)$ is compact, the metric $g$ is complete, and $(f,g)$ satisfies the \textit{Palais--Smale condition} \cite{palais-smale}. It is then canonically isomorphic to the relative singular homology
 \begin{align}
 HM_\ast (M,f) \cong H_\ast (M , \{ f < -\lambda \} ) \quad \text{for } \lambda \gg 0 .\label{isomorphism-Morse}
 \end{align}
 
Motivated by applications to Floer theory and enumerative geometry, discussed in \autoref{introduction:context}, this article extends this picture to a general situation when $\mathrm{Crit}(f)$ is non-compact and the Palais--Smale condition fails. In order to control the behavior of $f$ at infinity, we work in the following general framework:

\begin{itemize}
    \item[] \emph{(Compactification condition)} $M$ can be compactified to a compact manifold with corners $\overline M$ in a way compatible with $f$, that is:  there exists a \emph{boundary function} $\tau \colon \overline M \to \R_{\geq 0}$ vanishing on the lower strata such that $\overline f \coloneq \tau f$ extends to a smooth function on $\overline M$; see \autoref{definition:compactification} for details.
\end{itemize}

Given such a compactification, we then study a large perturbation of $f$ given by 
\[
    f_\varepsilon := f + \frac{\varepsilon}{\tau} : M \to \mathbb{R} \quad , \quad 0 < \varepsilon \ll 1,
\]
and assume that the critical values of $f_\varepsilon$ behave uniformly as $\varepsilon \to 0$:
\begin{itemize}
\item[] \emph{(Finite action condition)} The limits of convergent sequences of critical values of $f_\varepsilon$ as $\varepsilon \to 0$ are contained in a fixed interval $(-\lambda , \lambda )$, and $0$ is either a regular value or an isolated critical value of $\overline f|_{\partial\overline M}$ in the stratified sense; see \autoref{def:finite-type}. 
\end{itemize}

While the compactification condition is easy to verify in many geometric examples, the finite action condition is more elusive. However, using the theory of semialgebraic sets, we prove that it holds for compactifications that are \emph{locally semialgebraic}; see \autoref{sec:semialgebraic}. In particular, this applies to the complex algebraic setting. 

Using these assumptions, we show that for $0 < \varepsilon \ll 1$, the critical set of $f_\varepsilon$ is compact and the Palais--Smale property holds for $f_\varepsilon$ and a class of complete Riemannian metrics. This leads to the \textit{finite action Morse homology} $HM_*(M,f)$ defined as the homology of a chain complex generated by the critical points of $f_\varepsilon$ with $|f_\varepsilon| < \lambda$.

\begin{theorem}\label{theorem:main_real}
Let $(M, f)$ be a smooth manifold with a smooth function, equipped with a manifold with corners compactification as above. There is a canonical isomorphism \eqref{isomorphism-Morse} where $HM_\ast (M,f)$ is the finite action Morse homology. In particular, the finite action Morse homology of $(M,f)$ does not depend on the choice of compactification.
\end{theorem} 

The proof of \autoref{theorem:main_real} occupies the first three sections, and relies heavily on the framework of \emph{logarithmic calculus} on manifolds with corners \cite{melrose}. To take the limit $\varepsilon \to 0$ and establish an isomorphism with singular homology, we construct continuation maps using the \emph{slow} time-dependent gradient flow considered by Floer \cite{floer-spheres} and the geometry of the compactification. (This is a subtle point as the standard construction of the continuation map yields a map  from $\varepsilon$ to $\varepsilon'\leq \varepsilon$ which may fail to be an isomorphism.)

\subsection{Homology of vanishing cycles}\label{introduction:vanishing}
The finite action Morse homology groups have a  natural interpretation in an algebro-geometric context.
Consider a smooth complex algebraic variety $X$ equipped with a regular, i.e. complex algebraic, function $F : X \to \mathbb{C}$. By a result of Verdier \cite[Corollaire 5.1]{verdier}, there exist a minimal finite set $B(F) \subset \mathbb{C}$, called the set of \textit{bifurcation values} of $F$, such that 
\[
 F^{-1} (\mathbb{C}\setminus B(F) ) \xrightarrow{F} \mathbb{C} \setminus B(F)
\]
is a $C^\infty$ locally trivial fibration. Thus, a basic invariant of the pair $(X,F)$ is given by the singular homology of the pair $(X , F^{-1}(z_0) )$, for any $z_0 \in \mathbb{C}\setminus B(F)$. More canonically, the \emph{homology of vanishing cycles} of $(X, F)$ is the relative homology group
\[
H_{\ast}(X, F) \coloneq H_\ast (X , F^{-1} (-\lambda )) \cong H_\ast (X, \{\Re(F) \leq - \lambda \} ) \quad\text{for } \lambda \gg 0.
\]
From an algebraic viewpoint, a result of Sabbah \cite{sabbah} gives a canonical isomorphism 
\[
H^\ast (X , F)\otimes \mathbb{C}  \cong \mathbb{H}^\ast \big( X , (\Omega_{X}^\bullet , d + dF\wedge) \big)
\]
where the right-hand side is the \textit{algebraic deRham cohomology} of $(X,F)$: the hypercohomology of the complex of sheaves $(\Omega_{X}^\bullet , d + dF\wedge)$ on $X$ given by the algebraic differential forms with Witten's deformation of the exterior derivative \cite{witten-susy-morse-theory}.

If $F$ is a Morse function, the homology of vanishing cycles agrees with the Morse homology of $\Re(F)$.  Our goal is to provide a Morse-theoretic construction of the homology of vanishing cycles in a general situation. Non-compactness issues are especially prevalent in the algebraic case: for example, when $X$ is affine, then $\mathrm{Crit}(F)$ is non-compact whenever positive dimensional. The following is obtained as a special case of \autoref{theorem:main_real} by constructing a compactification of $(X,\mathrm{Re}(F))$ from a simple normal crossing compactification of $X$; this construction is outlined below. 
 
\begin{theorem} \label{theorem:main1}
    Let $X$ be a smooth complex algebraic variety with a regular function $ F: X \to \mathbb{C}$. There is a canonical isomorphism between the finite action Morse homology and the homology of vanishing cycles of $(X, F)$.
\end{theorem}

There is another homology theory naturally associated with the pair $(X,F)$, defined in terms of a perverse sheaf $\mathscr{P}_{F}^\bullet$ on the critical locus $\mathrm{Crit}F$ known as the \textit{perverse sheaf of vanishing cycles}. Denoting by $C(F) \subset \mathbb{C}$ the set of critical values, Brav, Bussi, Dupont, Joyce and Szendroi \cite{Brav2015,bussi} defined this perverse sheaf by
\[
\mathscr{P}^\bullet_F = \bigoplus_{z \in C(F)} (\phi_{F-z} \underline{\mathbb{Z}} )|_{\mathrm{Crit}F \cap F^{-1}(z)} 
\]
where $\phi_{F}$ denotes Deligne's functor of vanishing cycles \cite{deligne} . By the \textit{sheaf-theoretic cohomology of vanishing cycles} of $(X,F)$ we shall mean the hypercohomology group $HP^\ast (X, F) := \mathbb{H}^\ast (\mathrm{Crit}(F) , \mathscr{P}^\bullet_F )$. Describing this group via Morse homological constructions is an open problem of relevance to enumerative geometry, symplectic and low-dimensional topology \cite{bussi,abouzaidmanolescu}; see \cite{maxim22,maxim20} for progress in this direction.  \autoref{theorem:main1} provides a solution to this problem in the following situation. We say $F$ has \emph{no bifurcations at infinity} if for every $z \in \mathbb{C}$ there exists $\delta >0$ and a compact set $K \subset X$ such that the mapping $F : F^{-1} (B_\delta (z) )\setminus K \to B_\delta (z)$ is a $C^\infty$ locally trivial fibration. In this case, \autoref{theorem:main1} and  \autoref{proposition:sheafcomputation} imply:
\begin{corollary}\label{corollary:sheaf}If $F$ has no bifurcations at infinity, then there is an isomorphism between the finite action Morse cohomology and the sheaf-theoretic cohomology of vanishing cycles of $(X, F)$,
\[
HM^{\ast +1} (X,F) \cong HP^\ast (X, F).
\]
\end{corollary}

When $F$ has bifurcations at infinity, the isomorphism may no longer hold: 
\begin{example}
The function $F = x + x^2 y$ on $X = \mathbb{C}^2$ has $C(F) = \emptyset$ and hence $HP^\ast (X,F) = 0$. But one can see that $B(F) = \{ 0 \}$, so $F$ has a bifucation at infinity. By \autoref{theorem:main1}, one computes
\[
HM^\ast (X, F) \cong H^\ast (X,F) \cong  H^\ast (\mathbb{C}^2 ,\{ x+x^2y =-1\} ) \cong \mathbb{Z}[2] .
\]
\end{example}
\subsection{Outline of the construction}\label{introduction:sketch}

We now outline the construction of $HM_\ast (X,F)$ associated with a regular function $F \colon X \to \C$, assuming for simplicity that $X$ is quasi-projective. By Hironaka's resolution of singularities \cite{hironaka1,hironaka2}, there exists a smooth projective variety $Y$ with an embedding $X \subset Y$ whose complement $D := Y \setminus X$ is a simple normal crossings divisor. Let $D_i$ be the irreducible components of $D$, and denote by $\mathrm{ord}(F,D_i) \in \mathbb{Z}$ the order of $F$ along $D_i$. Choose positive integers $\alpha_i$ at least as large as the corresponding pole orders: $\alpha_i \geq - \mathrm{ord}(F,D_i )$. Each $D_i$ is the transversely cut-out zero set of a section $s_i$ of an algebraic line bundle $\mathscr{O}_Y (D_i )\to Y$, and we equip each $\mathscr{O}_Y (D_i ) $ with an algebraic Hermitian metric. Set 
\[
    \tau =  \prod_{i =1}^\ell |s_i |^{\alpha_i} \quad\text{and}\quad f_\varepsilon = \Re(F) + \frac{\varepsilon}{\tau}.
\]

We prove that $\tau$ and $f_\varepsilon$ extend smoothly to a manifold-with-corners compactification $\overline X$ obtained from $Y$ by taking the real blow-up along $D$. The crucial point is that $\overline X$ and $\tau$ satisfy the compactification and finite action conditions discussed earlier. This allows us to define a Morse complex
\[
    \big( CM_\ast^{(-\lambda , \lambda)}(X,(f_\varepsilon , g) , \partial \big), \quad 0 < \varepsilon \ll 1, 
\]
associated to a small perturbation of $f_\varepsilon$ for $0 < \varepsilon \ll 1$, and a suitable Riemannian metric $g$ adapted to the geometry of the compactification. This complex is generated only by critical points whose critical value lies in $(-\lambda , \lambda )$.
The \textit{finite action Morse homology} of $(X,F)$ is the homology of that complex
\[
HM_\ast (X,F) := H_\ast \big( CM_\ast^{(-\lambda , \lambda)} (X,f_\varepsilon ,g) , \partial \big), \quad 0 < \varepsilon \ll 1,
\]
which we then show to be isomorphic to the homology of vanishing cycles $H_*(X,F)$.

\begin{example}
If $X = \mathbb{C}^n$ and $F: \mathbb{C}^n \to \mathbb{C}$ is a polynomial of degree $d>0$, then $\Phi = 1/\tau$ is the strictly plurisubharmonic function $\Phi (x) = (1 + |x|^2 )^{d/2}$. The Kähler metric on $X$ associated to $\Phi$ (see, e.g. \cite{Cieliebak2012}) is then adapted to the compactification of $X= \mathbb{C}^n$ to $Y = \mathbb{C}P^n$. 
\end{example}

When $X$ is not quasi-projective then, by Nagata's compactification theorem \cite{nagata62,nagata63} combined with Hironaka's result, there still exists a simple normal crossings compactification $X \subset Y$ where $Y$ is a compact smooth complex algebraic variety, no longer projective. There is a subtle point here, in that we can no longer guarantee that the algebraic line bundles $\mathscr{O}_Y (D_i )$ carry \textit{algebraic} Hermitian metrics---however, we will see that they still carry \textit{Nash} Hermitian metrics, and with this modification the above construction of $HM_\ast (X, F)$ carries through. See \autoref{section:algebraic} for further details.


\subsection{Holomorphic Floer theories}\label{introduction:context}

Our construction of finite action Morse homology, allowing for noncompact or singular critical loci, is motivated by scenarios typically appearing in the study of \textit{holomorphic Floer theories}. This emerging framework, with origins in the Donaldson--Thomas program \cite{Donaldson1998}, aims to study the monodromy structures arising from holomorphic action functionals in Floer homology (\cite{Donaldson1998,Haydys2015,doan-rezchikov,kontsevich-soibelman}). This has inspired the development of new invariants of low-dimensional manifolds \cite{abouzaidmanolescu,Tanaka-Thomas}, holomorphic symplectic manifolds and their holomorphic Lagrangian submanifolds \cite{Brav2015,bussi} using algebro-geometric techniques---but their relation to Floer homology remains to be understood. Indeed, essential difficulties related to the failure of compactness typically arise when trying to set up the corresponding Floer homology with differential geometric constructions \cite{Taubes2013}.

Our approach to define Morse homology using a large perturbation $f_\varepsilon = f + \varepsilon/\tau$ has an infinite dimensional analogue in the setting of Floer homology of flat $SL(2,\mathbb{C})$ connections on a closed $3$-manifold (with the $SL(2,\mathbb{C})$ Chern--Simons functional playing the role of $F$), and the Vafa--Witten equations on closed $4$-manifolds. This has been studied by Taubes \cite{taubesVW,taubes2024spectral,taubes2024non} and by the second named author \cite{munoz25counting}, and constitutes the main motivation for the constructions in this article.

\subsection*{Acknowledgments} We thank Simon Donaldson, Johannes Huisman, Dominic Joyce, Francesco Lin, Laurentiu Maxim, and Richard Thomas for coversations from which this work has benefitted. Aleksander Doan is supported by Trinity College, Cambridge, and EPSRC grant 8927760 \emph{Gauge theory and pseudo-holomorphic curves in symplectic six-manifolds}.

\newpage
\section{Compactifications by manifolds with corners}
\label{section:compactifications}

Let $f \colon M \to \R$ be a smooth function on a manifold $M$. If $M$ is non-compact, then in general $f$ does not give rise to well-defined Morse homology, since its behavior at infinity can be highly irregular. In this section we introduce a general framework for controlling the behavior at infinity by assuming the existence of a compactification $M \subset \overline M$, where $\overline M$ is a compact manifold with corners and $f$ has poles along $\partial\overline M = \overline M \setminus M$. As we will see in \autoref{section:algebraic}, this assumption is satisfied in the algebraic situation discussed in \autoref{theorem:main1}. 

\begin{remark}
    A more general framework, arising naturally in algebraic geometry, would be to allow $\overline M$ to be a compact Whitney stratified space, with $f$ having poles along the lower-dimensional strata. However, every Whitney stratified space admits a canonical resolution by a manifold with corners, mapping the boundary to the lower-dimensional strata \cite[Section 2]{albin}. Therefore, this more general framework could be reduced to the one considered here.
\end{remark}

\subsection{Manifolds with corners}

We lay out our conventions for the notion of a smooth manifold with corners. For other notions appearing in the literature, see \cite[\S 2]{joyce-corners}. For $0 \leq m \leq n$, let 
\[
    \mathbb{R}^{n,m} \coloneq \mathbb{R}^{n-m}\times \mathbb{R}^{m}_{\geq 0}
\] 
be the \textit{$n$-corner of depth $m$}, equipped with the Euclidean topology. A smooth $n$--dimensional \emph{manifold with corners} $X$ is a Hausdorff, second-countable topological space with an equivalence class of smooth atlas of \emph{corner charts} $(U , \phi )$. Here, $U \subset \mathbb{R}^{n,m}$ is an open subset for some $0 \leq m \leq n$, and $\phi : U \hookrightarrow X$ is an open embedding. For any two charts $(U_1, \phi_1 )$, $(U_2 , \phi_2 )$ in a smooth atlas their transition map $\phi_{2}^{-1} \circ \phi_{1} : \phi_{1}^{-1}(U_1 \cap U_2 )  \to \phi_{2}^{-1} (U_1 \cap U_2 )$ is required to be smooth in the sense that it extends to a smooth map between open subsets of $\mathbb{R}^n$. By abuse of notation, we will often identify $U$ with $\phi(U)$ and simply say that  `$U \subset \R^{n,m}$ is a corner chart of $X$'. 

A map between manifolds with corners is smooth if in all corner charts it is smooth in the above sense. In particular, one can define the tangent bundle $TX \to X$ in the usual way, as well as the sheaf $\sO_X$ of smooth functions on $X$. We will use $TX$ interchangeably to refer to both the vector bundle or its underlying sheaf of $\sO_X$--modules; it shall be clear from context which we are referring to.

A manifold with corners carries a natural stratification by depth. The \textit{interior} of $X$ is the depth $0$ stratum; the \textit{boundary} $\partial X$ of $X$ is the union of strata of depth $\geq 1$. With this definition, $\partial X$ is only a stratified space with smooth strata and not a manifold with corners. (Another definition of the boundary is possible which makes it into a manifold with corners but we will not use it here \cite[\S 2]{joyce-corners}.)  However, $\partial X$ is equipped with a sheaf of smooth functions $\sO_{\partial X}$ defined as follows. Let $\sI_{\partial X} \subset \sO_X$ be the ideal subsheaf of functions vanishing along $\partial X$, and let $\iota: \partial X \to X$ be the inclusion. Define $\sO_{\partial X}$ as the inverse image 
\[
\sO_{\partial X} \coloneq \iota^{-1} (\sO_X /\sI_{\partial X} ).
\]
Concretely, for an open set $V \subset \partial X$, the ring $\sO_{\partial X}(V)$ consists of functions which extend to smooth functions on an open subset $U \subset X$ such that $U \cap \partial X = V$. The inclusion map is then a morphism of ringed spaces:
\[
\iota : ( \partial X , \sO_{\partial X} ) \to (X , \sO_{X} ).
\]

In addition, the \textit{conormal sheaf} of $\partial X$ is defined as 
\[
    N_{\partial X}^\ast = \iota^{-1} ( \sI_{\partial X}/\sI_{\partial X}^2) 
\]
which is naturally an $\sO_{\partial X}$--module. The \textit{normal sheaf} of $\partial X$ is then defined as the dual of $N_{\partial X}^\ast$ as an $\sO_{\partial X}$--module: 
\[
N_{\partial X} := \Hom_{\sO_{\partial X}} (    N_{\partial X}^\ast, \sO_{\partial X} ).
\]
Let $\iota^*TX = \iota^{-1}TX \otimes_{\iota^{-1}\sO_X} \sO_{\partial X}$ be the restriction of $TX$ to $\partial X$ thought of as an $\sO_{\partial X}$--module. 
There is a natural morphism of $\sO_{\partial X}$--modules $\iota^{\ast} TX \to  N_{\partial X}$ given by applying vector fields to functions in $\sI_{\partial X}$ and restricting to $\partial X$. The \textit{tangent sheaf} $T \partial X$ of $\partial X$ is defined as the kernel of this morphism:
\begin{equation}
 \label{normal}
    \begin{tikzcd}
    0 \ar{r} &  T \partial X \ar{r} & \iota^\ast TX \ar{r}&  N_{\partial X} \ar{r} & 0.
    \end{tikzcd}
\end{equation}
Concretely, for an open subset $V \subset \partial X$ we can identify $(T \partial X )(V)$ with sections over $V$ of the bundle $TX$ which are tangent to the strata of $\partial X$ and which admit an extension to smooth vector fields on an open subset $U \subset X$ with $U \cap \partial X = V$. In particular, if $S$ is a stratum of $\partial X$ and $\iota_{S} : S \to \partial X$ is the inclusion, then $\iota_S \colon (S,\sO_S) \to (\partial X, \sO_{\partial X})$ is a morphism of ringed spaces and applying the functor $\iota_{S}^\ast $ to \eqref{normal} recovers the familiar exact sequence of vector bundles for the submanifold $S \subset X$:
\[
    \begin{tikzcd}
    0 \ar{r} & TS \ar{r} & TX|_S \ar{r} & N_S \ar{r} & 0.
    \end{tikzcd}
\]

\begin{example}
    Let $X = \mathbb{R}^{n,m} = \mathbb{R}^{n-m}\times \mathbb{R}^{m}_{\geq 0}$ with coordinates $(x,y)$. Then, on any neighborhood of $(0,0) \in X$, the sheaf $T\partial X$ is generated by $\partial/\partial_{x_1}, \ldots , \partial/\partial_{x_{n-m}}$ and $N_{\partial X}$ by $\partial/\partial_{y_1} , \ldots , \partial/\partial_{y_m}$. In particular, if $X$ is a manifold with boundary, then the smooth vector bundles corresponding to the $\sO_{\partial X}$--modules $T\partial X$ and $N\partial X$ are the tangent and normal bundles of $\partial X$ as a submanifold.
\end{example}

\subsection{Logarithmic tangent bundle}

There is another notion of the tangent bundle to a manifold with corners, which will be useful in our discussion. We refer to \cite[\S 2.3]{joyce-tangent} for further details.

\begin{definition}
A \textit{logarithmic vector field} on $X$ is a section $V$ of $TX$ which is \emph{tangent to the boundary} in the following sense. For any corner chart $U \subset \mathbb{R}^{n,m}$ on $X$, and any $p \in U$, if $p$ lies in a stratum $S$ of $X$ then $V_p$ is tangent to $T_p S$. The \textit{logarithmic tangent bundle} (or \emph{$b$-tangent bundle} \cite{melrose}) is the vector bundle $\leftindex^{b} TX \to X$ associated to the sheaf of logarithmic vector fields on $X$, which is a locally free sheaf of $\mathscr{O}_X$--modules. 
\end{definition}

\begin{example}
Let $X = \mathbb{R}^{n,m} = \mathbb{R}^{n-m}\times \mathbb{R}^{m}_{\geq 0}$ with coordinates $(x,y)$. Then, on any neighborhood of $(0,0) \in X$, $\leftindex^{b} T\partial X$ is generated as an $\sO_X$--module by the logarithmic vector fields
\[
\frac{\partial}{\partial x_1} , \ldots , \frac{\partial}{\partial x_{n-m}} , y_1 \frac{\partial }{\partial y_1} , \ldots , y_m \frac{\partial }{\partial y_m}.
\]
\end{example}

As vector bundles, the ordinary and logarithmic tangent bundles are isomorphic, but \textit{not in a canonical fashion}. More precisely, there is a short exact sequence of $\mathscr{O}_X$--modules
\begin{equation}
    \label{ses1}
    \begin{tikzcd}
    0 \ar{r} & \leftindex^{b} {T} X \ar{r} & TX \ar{r}& \iota_\ast N_{\partial X} \ar{r}& 0
    \end{tikzcd}
\end{equation}
where $N_{\partial X}$ is the normal sheaf of $\partial X$. Since $i_*N_{\partial X}$ is supported on $\partial X$, it follows that $\leftindex^{b}{T}X$ and $TX$ are canonically isomorphic as vector bundles on $X\setminus\partial X$. This isomorphism can be extended to $X$, for example by retracting $X$ on a subset of $X\setminus\partial X$, but the extension is not canonical.

The restriction of $\leftindex^{b} {T} X$ to the boundary can be described as follows. If $S \subset \partial X$ is a stratum, then a logarithmic vector field on $X$ is tangent to $S$; hence there is a canonical short exact sequence of vector bundles
\begin{equation}
    \label{ses2}
    \begin{tikzcd}
    0 \ar{r} & \leftindex^{b} {N}_S \ar{r} & \leftindex^{b} {T}X|_{S} \ar{r} & TS \ar{r} & 0 
    \end{tikzcd}
\end{equation}
where $\leftindex^{b} {N}_S$, the \textit{logarithmic normal bundle} to $S$, is defined by this exact sequence. More generally, there is a natural morphism of $\mathscr{O}_{\partial X}$--modules $\iota^\ast \leftindex^{b} {T}X \to T \partial X$. The  \textit{logarithmic normal sheaf} $\leftindex^{b} {N}_{\partial X}$ to $\partial X$ is defined as the kernel of this morphism. Hence, it fits into a short exact sequence of $\mathscr{O}_{\partial X}$--modules 
\[
    \begin{tikzcd}
    0 \ar{r} & \leftindex^{b} {N}_{\partial X} \ar{r} & \iota^\ast \leftindex^{b} {T}X \ar{r} & T\partial X \ar{r} & 0.
    \end{tikzcd}
\]
This sequence recovers (\ref{ses2}) upon applying $\iota_{S}^\ast$ where $\iota_{S} : (S,\sO_S) \to (\partial X, \sO_{\partial_X})$ is the morphism of ringed spaces given by inclusion.

\begin{example}
    Let $X = \mathbb{R}^{n,m} = \mathbb{R}^{n-m}\times \mathbb{R}^{m}_{\geq 0}$ with coordinates $(x,y)$. Then, on any neighborhood of $(0,0) \in \partial X$, the sheaf $\leftindex^{b} N_{\partial X}$ is generated as an $\sO_{\partial X}$--module by 
    \[
     y_1 \frac{\partial }{\partial y_1} , \ldots , y_m \frac{\partial }{\partial y_m}.
    \]
\end{example}

Finally, the \textit{logarithmic cotangent bundle} of $X$ is the dual of the logarithmic tangent bundle, $\leftindex^{b} {T}^\ast X = (\leftindex^{b} {T}X )^\ast$, and its sections are the \textit{logarithmic $1$--forms}. Similarly, the sections of the $k$--th exterior power of $\leftindex^{b} {T}^\ast X$ are the logarithmic differential $k$--forms on $X$.

\begin{example}
    Let $X = \mathbb{R}^{n,m} = \mathbb{R}^{n-m}\times \mathbb{R}^{m}_{\geq 0}$ with coordinates $(x,y)$. Then, on any neighborhood of $(0,0) \in X$, $\leftindex^{b} T^\ast X$ is generated as an $\sO_X$--module by the logarithmic $1$-forms:
    \[
    d x_1 , \ldots , d x_{n-m} , \frac{dy_1 }{ y_1} , \ldots ,  \frac{d y_m }{ y_m}.
    \]
\end{example}

\subsection{Compactifications}

Let $M$ be a non-compact $n$--dimensional manifold and $f \colon M \to \R$ a smooth function. In favorable cases, we can tame the behavior of $f$ at infinity by compactifying $M$ to a manifold with corners $\overline M$ and controling the rate at which $f$ blows up along $\partial \overline M$.

\begin{definition}
\label{definition:compactification}
~
    \begin{enumerate}
    \item
    Let $X$ be manifold with corners. A \emph{boundary function} on $X$ is a smooth function $\tau \colon X \to [0,1]$ such that every point $p \in \partial X$ has a corner chart $U \subset \R^{n,m}$ centered at $p$ in which
    \begin{equation}
        \label{eq:boundaryfunction}
        \tau (x,y) = \prod_{i=1}^m y_i^{\alpha_i}
    \end{equation}
    where $(x,y)$ are coordinates on $\R^{n,m} = \mathbb{R}^{n-m}\times \mathbb{R}^{m}_{\geq 0}$ and
    $\alpha_1,\ldots,\alpha_m \in \N_{>0}$.
    \item
    Let $M$ be a smooth manifold and $f \colon M \to \R$ a smooth function.
    A \emph{compactification} of $(M,f)$ consists of a compact manifold with corners $\overline M$ and a smooth function $\overline f \colon \overline M \to \R$ such that $M$ is diffeomorphic to the interior of $\overline M$ and $\overline f / f$ extends to a boundary function $\tau$ on $\overline M$. 
    \item
    Given a compactification $(\overline M, \overline f)$ of $(M,f)$, a \emph{critical point at infinity} is a point $p \in \partial \overline M$ such that $\overline f(p) =0$ and $p$ is a critical point of $\overline f |_S$ where $S \subset \partial \overline M$ is a smooth stratum containing $p$, that is: $p$ is a stratified critical point of $\overline f|_{\partial \overline M}$ with critical value zero. 
    \end{enumerate}
\end{definition}

The significance of this definition lies in the following result.

\begin{proposition}\label{proposition:compactness}
    If $(\overline{M}, \overline f)$ is a compactification of $(M, f)$, then the union of the critical points of $f$ and critical points at infinity is a compact subset of $\overline M$.
\end{proposition}

\begin{corollary}
    If $(M,f)$ admits a compactification without critical points at infinity, then the set of critical points of $f$ is compact. 
\end{corollary}

\begin{proof}[Proof of \autoref{proposition:compactness}]
    The set of critical points at infinity is closed in $\partial\overline M$. Since $\overline M$ is compact, it remains to prove that if $(p_n)$ is a sequence of critical points of $f$ converging to a point $p \in \partial\overline M$, then $p$ is a critical point at infinity.
    
    Let $\tau = \overline f / f$ be the boundary function of the compactification. First, observe that $d\tau / \tau$ is a logarithmic $1$--form on $\overline M$. Indeed, in a corner chart $U \subset \mathbb{R}^{n,m}$ in which $\tau$ has the form (\ref{eq:boundaryfunction}), 
    \begin{align}
        \frac{d \tau}{\tau} = \sum_{i = 1}^{m} \alpha_i \frac{dy_i}{y_i}. \label{dtau/tau} 
    \end{align}
    On $M$, we have
    \begin{align}
        \tau df = d\overline{f} - \overline{f} \cdot \frac{d\tau}{\tau}\label{tau df}
    \end{align}
    which shows that $\tau df$ extends to logarithmic 1-form on $\overline M$. 
    Evaluating \eqref{tau df} at $p_n$ and taking the limit $n\to\infty$ yields
    \begin{equation}
        \label{critical at infinity}
        d\overline f(p) = 0 \quad\text\quad \overline f(p) = 0
    \end{equation}
    since $d\tau/\tau$ is nowhere vanishing as a logarithmic $1$--form by \eqref{dtau/tau}. 

    Let $S$ be the stratum of $\partial\overline M$ containing $p$. The restriction of $\tau df$ to $S$ is well-defined as a section of $\leftindex^{b}{T}^\ast \overline{M}|_{S}$ and given by 
    \begin{align}
        (\tau df)|_{S} = d (\overline{f}|_{S}) - \overline{f} \frac{d\tau}{\tau}|_{S}\label{restrict_S},
    \end{align}
    where $d ( \overline{f}|S ) \in \Gamma (S , T^\ast S )$ is regarded as an element of $\Gamma ( S , \leftindex^{b} {T}^\ast \overline{M}|_{S} )$ using the canonical bundle embedding $T^\ast S \hookrightarrow \leftindex^{b} {T}^\ast \overline{M}|_S$, which is obtained by dualising the map on the right-hand side of \eqref{ses2}.
    It follows from this and \eqref{critical at infinity} that $p$ is a critical point at infinity.
\end{proof}

\subsection{Riemannian metrics and the Palais--Smale condition}
\label{subsection: palais-smale}

Morse theory requires a choice of a Riemannian metric. When the underlying manifold is non-compact, the \emph{Palais--Smale condition} plays a key role in proving compactness for the moduli spaces of gradient trajectories. 

\begin{definition}
    A function $f \colon M \to \R$ and a complete Riemannian metric $g$ a non-compact manifold $M$ are said to satisfy the \text{Palais--Smale condition} if every sequence $(p_n)$ in $M$ such that $f(p_n)$ is bounded and $|\nabla f(p_n)| \to 0$ has a convergent subsequence. 
\end{definition}

Let $(\overline{M},\overline f)$ be a compactification of $(M, f)$, with boundary function $\tau = \overline{f}/f$. We will construct an appropriate Riemannian metric on $M$ from a logarithmic metric on $\overline M$. 

\begin{definition}
A \textit{logarithmic Riemannian metric} on $\overline{M}$ is a smooth bundle metric $\leftindex^{b} {g}$ on the logarithmic tangent bundle $\leftindex^{b} {T}\overline{M}$.
\end{definition}

Since $\leftindex^{b} {T}\overline{M}$ and $T\overline M$ are canonically isomorphic over $M$, a logarithmic Riemannian metric $\leftindex^{b} {g}$ on $\overline{M}$ yields an ordinary Riemannian metric on $M$. It will be more convenient to consider its \textit{associated cone metric} $g = \tau^{-1} \cdot \leftindex^{b} {g}$, which is also a Riemannian metric on $M$.

\begin{example}
    \label{example:logarithmic metric}
    Let $X = \mathbb{R}^{n,m} = \mathbb{R}^{n-m}\times \mathbb{R}^{m}_{\geq 0}$ with coordinates $(x,y)$. Then, on any neighborhood of $(0,0) \in X$, a logarithmic Riemannian metric has the form
    \[
        \leftindex^{b} {g} = \sum_{i,j =1}^{n-m} a_{ij}  dx_i \otimes dx_j  + \sum_{i=1}^{n-m}\sum_{j =1}^m b_{ij}  \left( dx_i \otimes \frac{dy_j}{y_j} + \frac{dy_j}{y_j}  \otimes dx_i  \right) + \sum_{i,j=1}^{m} c_{ij} \frac{dy_i}{y_i} \otimes \frac{dy_j}{y_j} 
    \]
    for smooth coefficient functions $a_{ij}$, $b_{ij}$, $c_{ij}$ on $\R^{n,m}$, with $a_{ij}$ and $c_{ij}$ symmetric. The non-degeneracy of $\leftindex^{b} {g}$ implies that the matrices $(a_ij)$ and $(c_ij)$ are non-degenerate. 
\end{example}

\begin{example}
    Let $X$ be a compact manifold with boundary. Any Riemannian metric on $X\setminus\partial X$ which on the collar $(0,1) \times \partial X$ has the form
    \[
        \leftindex^{b} {g}  = g_{\partial X} + \frac{1}{y^2} dy^2,
    \]
    where $g_{\partial X}$ is a Riemannian metric on $\partial X$ and $y \in [0,1)$ is the coordinate on the collar, extends to a logarithmic Riemannian metric on $X$. If $\tau$ is a boundary function of the form $\tau = y$ on the collar, then the associated cone metric is
    \[
        g = \tau^{-1} \cdot \leftindex^{b} {g}  = \frac{1}{y} g_{\partial X} + \frac{1}{y^3} dy^2.
    \]
    After change of coordinates $y = r^{-2}$, 
    \[
        g = r^2 g_{\partial X} + dr^2,
    \]
    which shows that $g$ is a Riemannian metric on $X\setminus\partial X$ with a conical end $r \to \infty$. 
\end{example}

\begin{proposition}\label{proposition:PS}
    Let $(\overline{M}, \overline{f})$ be a compactification of $(M , f)$ with boundary function $\tau = \overline f / f$. Let $\leftindex^{b} {g}$ be a logarithmic Riemannian metric on $\overline{M}$.  
    \begin{enumerate}
        \item Both $\leftindex^{b} {g}$ and its associated cone metric $g = \tau^{-1} \cdot \leftindex^{b} {g}$ are complete Riemannian metrics on $M$. 
        \item If $f$ has no critical points at infinity, then the critical locus of $f$ is compact and the pair $(f,g)$ satisfies the Palais--Smale condition on $M$.
    \end{enumerate}
\end{proposition}

\begin{proof}
    (1) To prove completeness, let $\gamma \colon [0,T) \to M$ be a smooth path such that $\gamma(t)$ has no limit in $M$ as $t\to T$. We need to show that $\gamma$ has infinite length with respect to both $\leftindex^{b} {g}$ and $g$. Since $g = \tau^{-1} \cdot \leftindex^{b} {g} \gtrsim \leftindex^{b} {g} $ on $M$, it suffices to prove this only for $\leftindex^{b} {g}$.
    
    Since $\overline M$ is compact, there is a sequence $t_n$ converging to $T$ such that $\gamma(t_n)$ converges to a point $p \in \partial \overline M$. Let $U \subset \R^{m,n}$ be a corner chart centered at $p$ and let $g$ be an admissible metric of the form described in \autoref{example:logarithmic metric}. 
    For $t$ close to $T$, $\gamma(t)$ is contained in $U$, in which case we denote $\gamma(t) = (x(t), y(t))$, where $(x,y)$ are coordinates on $\R^{m,n}$. Note that  $(x(t_)n),y(t_n)) \to (0,0)$ as $n \to \infty$. 
    By the Cauchy--Schwarz inequality, there is $s_0 < T$ such that for $t \in [s_0, T)$, 
    \[
        \leftindex^{b}{g}(\dot\gamma(t), \dot\gamma(t))^{1/2} \geq c \sum_{i=1}^k \frac{ \dot y_i(t)}{y_i(t)},
    \]
    so that the length of $\gamma \colon [s_0,s] \to U$ is
    \[
        \int_{s_0}^s   \leftindex^{b} {g} (\dot\gamma(t), \dot\gamma(t))^{1/2} \rd t \geq c   \sum_{i=1}^k \int_0^s \frac{ \dot y_i(t)}{y_i(t)} \rd t \geq c \sum_{i=1}^k \ln( y_i(s) / y_i(s_0) ).
    \]
    The right-hand side diverges to infinity for $s=t_n$ as $n\to\infty$. Therefore, $\gamma$ has infinite length with respect to $\leftindex^{b} {g}$, as we wanted to prove. 

    (2) To verify the Palais--Smale condition for $(f,g) $ in the absence of critical points at infinity, note that the gradient $\nabla f$ with respect to $g = \tau^{-1} \leftindex^{b} {g}$ is given by $\nabla f = \tau \leftindex^{b} {\nabla}f $ where $\leftindex^{b} {\nabla}f$ is the gradient with respect to $\leftindex^{b} {g}$. This vector field is dual to the logarithmic $1$--form $\tau df$ with respect to $\leftindex^{b} g$, see \eqref{tau df}. Hence $\nabla f $ is a logarithmic vector field. On $M$, we have
    \begin{align}
        | \nabla f|^2 = \tau^{-1} \cdot \leftindex^{b}| \nabla f|^2 \label{norm2}
    \end{align}
    Let $p_n  \in M$ be a sequence with $f(p_n )$ bounded and $|\nabla f (p_n ) |\to 0$. Suppose that no subsequence of $p_n$ converges in $M$. Then a subsequence must converge to a point $p \in S$ for some stratum of $\partial\overline{M}$. Evaluating \eqref{norm2} at $p_n$ and taking the limit $n\to\infty$ shows that $\leftindex^{b}| \nabla f(p)| =0$. In other words, the logarithmic vector field $\nabla f$, as a smooth section of $\leftindex^{b}{T}\overline M$ over $\overline M$,  vanishes at $p$. In turn, by (\ref{restrict_S}) this says that $p$ is a critical point of $f$ at infinity, which is a contradiction. Thus, a subsequence of $p_n$ converges in $M$, which establishes the Palais--Smale condition.
\end{proof}
\section{Morse homology in the absence of critical points at infinity}
\label{section:morsehomology}

Let $f \colon M \to \R$ be a smooth function on a non-compact manifold. Using the framework from the previous section, it readily follows that when $(M,f)$ admits a compactification without critical points at infinity, in the sense of \autoref{definition:compactification}, then the Morse homology of (a small perturbation of) $f$ is well defined and isomorphic to the singular homology $H_\ast (M , \{ f \leq -c\})$ for $c \gg 0$. 

In this section we shall study, more generally, the Morse homology groups for such functions $f$ constructed only using critical points in a given action window $(a,b) \subset \mathbb{R}$, where $a,b$ are regular values of $f$. Furthermore, to a homotopy of functions of this kind we associate continuation maps between Morse homology groups. This will be crucial for treating the more delicate case when $f$ has critical points at infinity, which is the main goal of this article. The construction of continuation maps in action windows requires some care, and for this we will make crucial use of the framework from the previous section and certain `slow' time dependent gradient equation (also considered by Floer in \cite{floer-spheres}).

\subsection{Parameter space}

Throughout this section, let $\overline M$ be a compact manifold with corners and let $M = \overline M \setminus \partial\overline M$ be the interior of $\overline M$. Fix a boundary function $\tau \colon \overline M \to [0,1]$, as in \autoref{definition:compactification}, and  extended real numbers $a , b \in\mathbb{R}\cup\{\pm \infty\}$ with $a < b$.

\begin{definition}
    \label{definition: parameters}
    Denote by $\mathscr{P}_{(a,b)}$ the Frechét space of pairs $(f,g)$ where $f : M \to \mathbb{R}$ is a smooth function on $M$ and $g$ is a Riemannian metric on $M$, satisfying the following:
    \begin{enumerate}
        \item $\overline{f} \coloneq \tau f$ extends to a smooth function on $\overline{M}$;
        \item $g =  \tau^{-1}\cdot \leftindex^b {g}$ is the cone metric associated with a logarithmic metric $\leftindex^b {g}$ on $\overline{M}$;
        \item the compactification $(\overline M, \overline f)$ of $(M,f)$ has no critical points at infinity; 
        \item $a$ and $b$ are either infinite or regular values of $f$.
    \end{enumerate}
    Denote by $\mathscr{P}_{(a,b)}^\reg \subset \mathscr{P}_{(a,b)}$ the \textit{regular locus} defined as the subspace consisting of pairs $(f,g)$ which, in addition to the four points above, satisfy:
    \begin{enumerate}
    \item[5.] $(f,g)$ satisfies the Morse--Smale condition in the action window $\{ a < f < b \}$.
    \end{enumerate}
\end{definition}

The following is a standard application of the Sard--Smale theorem.

\begin{proposition}
    $\mathscr{P}_{(a,b)}^\reg$ is dense in $\mathscr{P}_{(a,b)}$.
\end{proposition}

In what follows, we explain how to associate a Morse homology group $HM_{\ast}^{(a,b)}(f,g)$ to each pair $(f,g) \in \mathscr{P}_{(a,b)}$, constructed from the critical points of a small perturbation of $f$ in the action window $\{ a < f < b \}$. Moreover, we associate to a path $(f_t , g_t )$ in $\mathscr{P}_{(a,b)}$ a continuation map $HM_{\ast}^{(a,b)}(f_0 , g_0 ) \to HM_{\ast}^{(a,b)}(f_1 , g_1 )$ which is functorial under concatenation of paths.

\subsection{Morse complex in an action window}

Let $(f,g) \in \mathscr{P}_{(a,b)}^{reg}$. Define a finitely-generated chain complex with coefficients in $\mathbb{Z}$
\[
    ( C_{\ast}^{(a,b)} (f,g) , \partial )
\]
referred to as the \textit{Morse complex of $(f,g)$ in the action window} $(a,b)$, as follows. As a graded abelian group, $C_{\ast}^{(a,b)}(f,g)$ is freely generated by the critical points of $f$ in $\{ a < f < b \}$,
\[
 C_{\ast}^{(a,b )} ( f , g) = \bigoplus_{x \in \mathrm{Crit}f \cap \{a < f < b \}} \mathbb{Z},
\]
and the grading of a critical point $x$ is defined to be the Morse index $\mathrm{ind}(f , x )$. For $x_-,x_+ \in \mathrm{Crit}f \cap \{ a < f < b \}$ such that $\mathrm{ind}(f , x_- ) = \mathrm{ind}(f , x_+ ) + 1$, the coefficient of $\partial x_-$ in $x_+$ is given by counting unparametrised negative gradient trajectories of $f$ with respect to $g$ connecting $x_-$ and $x_+$:
\[
\langle \partial x_- , x_+ \rangle = \# \Big( \mathscr{M}_{(f,g)} (x_-,x_+) / \mathbb{R} \Big) \in \mathbb{Z}.
\]
Here, $\mathscr{M}_{(f,g)} (x_-,x_+)$ is the \textit{moduli space of parametrised negative gradient trajectories}
\[
\mathscr{M}_{(f,g)} (x_-,x_+) := \Big\{ x : \mathbb{R}\to M \, | x^\prime (s)+ (\nabla_{g} f) (x(s))= 0 \, \text{  and } \lim_{s\to\pm \infty} x(s) = x_\pm \Big\},
\]
and $\#$ denotes the signed count of points in the \textit{unparametrised moduli space} $\mathscr{M}_{(f,g)}(x_- , x_+ )/\mathbb{R} $ obtained by quotienting by the translation $\mathbb{R}$--action.

\begin{remark}
It is well-known that the zero-dimensional manifold $\mathscr{M}_{(f,g)}(x_- , x_+ )/\mathbb{R}$ acquires a preferred orientation (i.e. a sign at each point) once an orientation for each of the unstable manifolds of the critical points of $f$ has been fixed (see e.g. \cite{schwarz-book}). Thus, we need to make this choice in order to define $\# \big( \mathscr{M}_{(f,g)}(x_- , x_+ )/\mathbb{R}\big)$. Alternatively, one can modify the definition of $(C_\ast^{(a,b)} , \partial)$ so that no such choices are needed, as in \cite[\S 2.2]{KM}. Similar comments apply to all counts of trajectory spaces throughout the article, and we shall make no further mention to this matter.
\end{remark}

\begin{proposition}
    \label{proposition:morse homology}
   $( C_{\ast}^{(a,b)} (f,g) , \partial )$ is a chain complex whose homology
    \[
        HM_{\ast}^{(a,b)} (f,g) = H_\ast \big( C_\ast^{(a,b)}(f,g) , \partial \big) .
    \] 
    is canonically isomorphic, as a graded abelian group, to the relative singular homology
    \[
        HM_\ast^{(a,b)}(f,g) \cong H_\ast (\{ f \leq b \} , \{f\leq a \} ; \mathbb{Z}).
    \]
\end{proposition}

We will refer to $ HM_\ast^{(a,b)}(f,g)$ as the \textit{Morse homology in the action window $(a,b)$}.  

\begin{proof}
Given the results of the previous section, this is a standard result in Morse theory on non-compact manifolds (see for example, \cite[Theorem 2.8, Proposition 2.15]{abbondandolo} or \cite{schwarz-book}). The following outlines the argument. First, using the Palais--Smale condition of $(f,g)$ (Proposition \ref{proposition:PS}) one can show as in \cite[Lemmas 2.38 and 2.39]{schwarz-book} that any sequence of trajectories $( x_n )$ in $\mathscr{M}_{(f,g)}(x_- , x_+ )$ converges, after passing to a subsequence, in $C^{k}_{loc}$ for all $k$ to another gradient trajectory (converging to different limits, possibly). Using this result, one next defines the usual compactification of $\mathscr{M}_{(f,g)}(x_- , x_+ )/\mathbb{R}$ by broken trajectories, denoted $\overline{\mathscr{M}_{(f,g)}(x_- , x_+ )/\mathbb{R}}$. Standard gluing results show that this compactification is a smooth manifold with corners, whose codimension $0$ stratum (i.e. its interior) is $\mathscr{M}_{(f,g)}(x_- , x_+ )/\mathbb{R}$ and whose codimension $c$ stratum consists of the $c$ times broken trajectories from $x_- $ to $x_+$. Thus, when $\mathrm{ind}(f,x_- ) = \mathrm{ind}(f,x_+ )+1$, then $\mathscr{M}_{(f,g)}(x_- , x_+ )/\mathbb{R}$ is compact and zero-dimensional and thus finite---showing $\partial$ is well-defined. When $\mathrm{ind}(f,x_-) = \mathrm{ind}(f,x_+ )+2$ then $\overline{\mathscr{M}_{(f,g)}(x_- , x_+ )/\mathbb{R}}$ is a compact $1$-manifold with boundary (whose boundary points are the once-broken trajectories), and enumerating boundary points in these moduli spaces yields the identity $\partial^2 = 0$.
\end{proof}

\subsection{Continuation maps in an action window}

Let $(f_t , g_t ) \in \mathscr{P}_{(a,b)}$ be a path with endpoints $(f_0 , g_0 )$ and $(f_1 , g_1 )$ in $\mathscr{P}_{(a,b)}^{reg}$. To such a path we associate a continuation isomorphism in Morse homology in the action window $(a,b)$. The existence of this map and the fact that it is an isomorphism is not guaranteed by general theory, and we shall crucially use here the `slow' time-dependent gradient equation considered by Floer in \cite{floer-spheres}.

We first make some preliminary observations. Fix a smooth non-decreasing function $\gamma : \R \to [0,1]$ such that $\gamma (s) = 0$ for $s\leq -1$ and $\gamma (s) = 1$ for $s \geq 1$. Fix also a \emph{slowing parameter} $\delta > 0$, which later we will assume to be small. Consider a smooth path $x : \mathbb{R} \to M$ such that $x(s)$ converges exponentially to critical points of $f_0$ and $f_1$ as $s \to -\infty$ and $s \to + \infty$. The \textit{analytic energy} $\mathcal{E}^{an}(x) \in \mathbb{R}$ and \textit{topological energy} $\mathcal{E}^{top}(x) \in \mathbb{R}$ of the path $x$ are defined by
\begin{align*}
    \mathcal{E}^{an}(x) & := \int_{-\infty}^\infty | x^\prime (s)|^2 + |(\nabla_{g_{\gamma(\delta s)}} f_{\gamma(\delta s)} )(x(s))|^2 - 2 (\partial_s f_{\gamma(\delta s)} )(x(s)) \, ds \\
    \mathcal{E}^{top}(x) & := 2\cdot \big( f_0(x(-\infty )) - f_1(x(+\infty))\big).
\end{align*}
We have an energy identity
\begin{align*}
    \mathcal{E}^{an} (x) = \mathcal{E}^{top}(x) + \int_{-\infty}^{\infty} | x^\prime (s) + (\nabla f_{\gamma (\delta s)} )(x(s))|^2  \, ds
\end{align*}
Thus, the equality $\mathcal{E}^{an}(x) = \mathcal{E}^{top} (x)$ is attained if and only if $x$ solves the $\delta$--slow time-dependent gradient equation
\begin{align}
    x^\prime (s) + (\nabla_{g_{\gamma (\delta s )}} f_{\gamma(\delta s )} )(x(s)) = 0 \quad\text{for all } s \in \R.\label{timedependentgradienteq}
\end{align}
In this case,  we denote both $\mathcal{E}^{an}(x)$ and $\mathcal{E}^{top}(x)$ simply by $\mathcal{E}(x)$.\\

The following lemma allows us to control the behavior of solutions to (\ref{timedependentgradienteq}):

\begin{lemma}\label{lemma:nocrossing}
    Let $(f_t , g_t ) \in \mathscr{P}_{(a,b)}$ be a path with endpoints $(f_0 , g_0 )$ and $(f_1 , g_1 )$ in $\mathscr{P}_{(a,b)}^{reg}$. Fix $\sigma >0$ such that $f_t$ has no critical values in $[a-\sigma ,a]$ and $[b,b+\sigma]$ for all $t \in [0,1]$. There exists $\delta_0 >0 $ with the following property. If $0 < \delta \leq \delta_0$ and $x : \mathbb{R}\to M$ is a solution to \eqref{timedependentgradienteq} converging to critical points of $f_0$ and $f_1$ as $s \to -\infty$ and $s\to +\infty$, then:
    \begin{enumerate}
        \item If $f_0 (x(-\infty) ) < b$ then $f_{\gamma(\delta s)}(x(s)) \leq b + \sigma$ for all $s$ and, in particular, $f_1 (x(+\infty)) < b$.
        \item If $f_1 (x(+\infty)) >a$ then $f_{\gamma(\delta s)} (x(s)) \geq a-\sigma$ for all $s$ and, in particular, $f_0 (x(-\infty)) > a$.
    \end{enumerate}
\end{lemma}

\begin{proof}
    For every $t$, the function $f_t$ has no critical values in $[b,b+\sigma]$ and no critical points at infinity. Therefore, by \eqref{norm2}, on $\{ b +\sigma/2 \leq f_t \leq b + \sigma\} \subset M $ we have
    \[
    |\nabla_{g_t} f_t (x)|^2 \geq \frac{c}{\tau}
    \]
    for a constant $c$ which is independent of $t$. On the other hand, since $f_t = \overline{f}_t /\tau$ where $\overline{f}_t $ is a path of smooth functions on the compact manifold with corners $\overline{M}$, we have on $M$
    \[
    |\partial_t f_t (x) | \leq \frac{c^\prime}{\tau} .
    \]
    These two estimates combined show that for sufficiently small $\delta$ the following holds. For all $s\in \mathbb{R}$ and $x\in \{ b +\sigma/2 \leq f_t \leq b + \sigma\}  $,
    \begin{align}
    |\nabla_{g_{\gamma(\delta s)}} f_{\gamma(\delta s)} (x) |^2 - 2 ( \partial_s f_{\gamma(\delta s)} )(x) \geq 0 .\label{inequality_strip}
    \end{align}
    
    Let $x : \mathbb{R} \to M$ be a path as in the statement above, and suppose for a contradiction that $f_0 (x(-\infty)) < b$ and $f_{\gamma(\delta s)} (x(s)) > b + \sigma$ at some time. Then there exist times $s_1 < s_2$ such that
    \begin{itemize}
        \item $f_{\gamma(\delta s_1 )} (x(s_1 )) = b + \sigma/2$, 
        \item $f_{\gamma(\delta s_2 )} (x(s_2 )) = b + \sigma$, 
        \item $f_{\gamma(\delta s )} (x(s )) \in [b + \sigma/2 , b + \sigma ]$ for all $s \in [s_1 , s_2 ]$.
    \end{itemize}
    Therefore, by the energy identity,
    \begin{align*}
    0  &= \int_{s_1}^{s_2} | x^\prime (s) + (\nabla f_{\gamma(\delta s)} (x(s))|^2  \, ds \\
     &= \int_{s_1}^{s_2} \Big( | x^\prime (s)|^2  + |(\nabla f_{\gamma(\delta s)} (x(s))|^2  - 2 (\partial_s f_{\gamma(\delta s)}  )(x(s))\Big) \, ds \\
       & \quad - 2\cdot (f_{\gamma(\delta s_1)}(x(s_1)) - f_{\gamma(\delta s_2 )}(x(s_2) ) \\
       &\geq 0 + \sigma > 0,
    \end{align*}
    where we have used \eqref{inequality_strip}. This contradiction proves the first assertion. The second assertion is proved in the same way. 
\end{proof}

In what follows, fix $\delta$ with $0< \delta\leq \delta_0$ as \autoref{lemma:nocrossing}. For critical points $x_- \in \mathrm{Crit}f_0$ and $x_+ \in \mathrm{Crit}f_1$ with $f_0 (x_- ) , f_1 (x_+ ) \in (a,b)$, consider the moduli space of trajectories 
\[
\mathscr{M}_{\{(f_t , g_t ) \}}(x_- , x_+ ) = \Big\{ x : \mathbb{R} \to M \, | \, x(s) \text{ solves (\ref{timedependentgradienteq}) and } \lim_{s \to \pm \infty}x(s) = x_{\pm} \Big\} .
\]
Assume that the path $(f_t , g_t )$ is generic, so that the moduli space $\mathscr{M}_{\{(f_t , g_t)\}}(x_- , x_+ )$ is a smooth manifold of dimension
\[
    \dim\mathscr{M}_{\{(f_t , g_t)\}}(x_- , x_+ ) = \mathrm{ind}(f_0 , x_- ) - \mathrm{ind}(f_1 , x_+ ).
\]



Using the Palais--Smale condition of $(f_0 , g_0 )$ and $(f_1 , g_1 )$, one shows using standard arguments (as outlined in the proof of Proposition \ref{proposition:morse homology}) that the moduli space $\mathscr{M}_{\{ (f_t , g_t )\}}(x_- , x_+ )$ has a natural compactification by broken trajectories to a smooth manifold with corners $\overline{\mathscr{M}}_{\{ (f_t , g_t )\}}(x_- , x_+ )$: 
\begin{align*}
\overline{\mathscr{M}}_{\{ (f_t , g_t )\}}(x_- , x_+ ) &= 
\Big( \bigcup_{y_ \in \mathrm{Crit}f_0} \overline{\mathscr{M}_{\{ (f_0 , g_0 )\}}(x_- , y_- )/\mathbb{R}} \times \mathscr{M}_{(f_t , g_t )}(y_- , x_+ )\Big)\\
& \cup \Big( \bigcup_{y_+ \in \mathrm{Crit}f_1}   \mathscr{M}_{(f_t , g_t )}(x_- , y_+ ) \times \overline{\mathscr{M}_{\{ (f_1 , g_1 )\}}(y_+ , x_+ )/\mathbb{R}} 
\Big).
\end{align*}
In particular, when its dimension is $0$, then $\mathscr{M}_{\{(f_t , g_t)\}}(x_- , x_+ )$ is a finite set of points; when its dimension is $1$, then $\overline{\mathscr{M}}_{\{(f_t , g_t)\}}(x_- , x_+ )$ is a compact $1$-dimensional manifold with boundary, whose boundary is given by the once-broken trajectories from $x_-$ to $x_+$.
We thus obtain a degree-preserving map of graded abelian groups \begin{gather*}
    c : C_{\ast}^{(a,b)}(f_0 , g_0) \to C_\ast^{(a,b)}(f_1 , g_1) \\
    c(x_- ) = \sum_{x_+} \# \mathscr{M}_{\{ (f_t , g_t )\}}(x_-,x_+) \cdot x_+ ,
\end{gather*}
where the sum is taken over all critical points $x_+$ of $f_1$ in the action window $(a,b)$ and satisfying $\mathrm{ind}(f_0 , x_- ) - \mathrm{ind}(f_1 , x_+ ) = 1$, and $\#$ denotes the signed count of points with respect to the natural orientation on the moduli space.

\begin{proposition}
    The map $c$ is a chain map. 
\end{proposition}

\begin{proof}
    Let $x_- , x_+$ be critical points of $f_0$ and $f_1$ respectively, with $\mathrm{ind}(f_0 , x_- )-  \mathrm{ind}(f_1 , x_+ ) = 1$. The $1$--dimensional moduli space $\mathscr{M}_{\{(f_t , g_t )\} }(x_- , x_+ )$ is compactified to $\overline{\mathscr{M}}_{\{(f_t , g_t)\}}(x_- , x_+ )$ by adding once-broken trajectories of the following two kinds:
    \begin{itemize}
    \item $(x_0 , x_{01})$ where $x_0$ is a flowline of $-\nabla_{g_0} f_0$ from $x_-$ to a critical point $y_- \in \mathrm{Crit}f_0$, and $x_{01} \in \mathscr{M}_{\{(f_t , g_t)\}}(y_- , x_+ )$.
    \item $(x_{01} , x_{1})$ where $x_{01} \in \mathscr{M}_{\{(f_t , g_t )\}}(x_- , y_+)$ for some critical point $y_+ \in \mathrm{Crit}f_1$, and $x_1$ is a flowline of $-\nabla_{g_1} f_1$ from $y_+$ to $x_+$.
    \end{itemize}
    From this, the standard argument involving enumerating boundary points yields the chain map identity $\partial_1 c = c \partial_0$, \textit{provided} that the critical points $y_- \in \mathrm{Crit}f_{0}$ and $y_+ \in \mathrm{Crit}f_1$ have critical values in the interval $(a , b)$. To establish the latter, observe that $f_0 (x_- ) < b$ and \autoref{lemma:nocrossing} imply that $f_1 (y_+ ) <b  $.
    On the other hand, if $f_1 (y_+ ) \leq a$, then
    \[
        0 < \mathcal{E}(x_1 ) = 2 (f_1 (y_+ ) - f_1 (x_+ )) < 2 (a -a ) = 0,
    \] 
    which is impossible. Thus, $f_1 (y_+ ) \in (a,b)$. The proof for $y_-$ is analogous.
\end{proof}

Standard homotopy arguments prove the following.

\begin{proposition}
\label{proposition: continuation}
~  
\begin{enumerate}
    \item Up to chain homotopy, the chain map $c$  depends only on the homotopy class of the path $(f_t ,g_t )$ in $\mathscr{P}_{(a,b)}$ relative to the endpoints.
    \item The continuation maps behave functorially under concatenation of paths, that is: given $(f_0 , g_0 ), (f_1 , g_1 ) , (f_2 , g_2 ) \in \mathscr{P}_{(a,b)}^{reg}$ and paths $(f^{01}_t , g^{01}_t )$ and $(f_{t}^{12} , g_{t}^{12} )$ in $\mathscr{P}_{(a,b)}$ from $(f_0 , g_0)$ to $(f_1 , g_1 )$, and from $(f_1 , g_1 )$ to $(f_2 , g_2 )$ respectively, then the continuation chain map $c_{02}$ of the concatenation is chain homotopic to $c_{12} \circ c_{01}$. 
    \item The continuation map $c$ induces an isomorphism on Morse homology in the action window $(a,b)$:
     \[
        c \coloneq HM_{\ast}^{(a,b)} (f_0 , g_0 ) \xrightarrow{\cong} HM_\ast^{(a,b)}(f_1 ,g_1).
    \]
\end{enumerate}
\end{proposition}

Since $\mathscr{P}_{(a,b)}$ is locally contractible, it follows that the Morse homology $HM_{\ast}^{(a,b)}(f,g)$ is defined for all $(f,g) \in \mathscr{P}_{(a,b)}$ even when $(f,g) \notin \mathscr{P}_{(a,b)}^{reg}$, and that continuation maps associated to arbitrary paths in $\mathscr{P}_{(a,b)}$ are also defined. Therefore, we have the following; for details, see \cite[\S 2.3]{hutchings-families} and \cite[\S 2.7]{abbondandolo}.
 
\begin{proposition}\label{proposition:localsystem}
    The Morse homology groups $HM_\ast^{(a,b)}(f,g)$ form a local system of finitely-generated graded abelian groups over the space $\mathscr{P}_{(a,b)}$. Under the canonical isomorphism from \autoref{proposition:morse homology}, this local system agrees with the local system of singular homology groups $H_\ast (\{ f \leq b \} , \{f \leq a \} ;\mathbb{Z})$.
\end{proposition}

When $M$ is compact, $a = -\infty$ and $b = +\infty$ then this local system is trivial, since $\mathscr{P}_{(a,b)}$ is contractible. However, in general this local system is non-trivial. 
\section{Finite action Morse homology}
\label{section:finiteaction}

Let $(\overline M, \overline f)$ be a compactification of $(M,f)$. In the presence of critical points at infinity, the construction of Morse homology of the pair $(f,g)$ from the previous section does no longer apply: the critical locus of $f$ may be non-compact and the Palais--Smale property of $(f,g)$ is no longer guaranteed. In this section, we carry out a construction of Morse homology groups in this more general situation by considering certain `large' perturbation of $f$ provided by the geometry of the compactification.


Namely, we consider a perturbation of $f$ naturally associated with the boundary function $\tau = \overline f/f$ of the compactification $(\overline M,\overline f)$:
\begin{equation}
    \label{tau perturbation}
    f_\varepsilon = f + \varepsilon/\tau.
\end{equation}
We show that under a finiteness assumption on the compactification and for $0< \varepsilon \ll 1$, the function $f_\varepsilon$ has no critical points at infinity and its Morse complex contains a subcomplex generated by critical points with bounded action as $\varepsilon \to 0$. The Morse homology of this subcomplex is independent on the various choices and recovers the singular homology group $H_*(M, \{ f < -\lambda \}; \Z)$ with $\lambda \gg 0$.

\subsection{Finite action compactifications}

The following notion is key in controlling the critical points of $f_\varepsilon$ as $\varepsilon \to 0$. 

\begin{definition}
    \label{def:finite-type}
    Let $(\overline M,\overline f)$ be a compactification of $(M,f)$.  Let $\cC_\varepsilon$ be the set of critical values of $f_\varepsilon$ defined in \eqref{tau perturbation}. Define
    $\cC \subset \R$ to be the set of limits of convergent sequences $(c_n)$ with $c_n \in \cC_{\varepsilon_n}$ for some sequence $(\varepsilon_n)$ with $\lim_{n\to\infty} \varepsilon_n = 0$. In particular, $\cC_0 \subset \cC$. We say that $(\overline M, \overline f)$ is a \emph{finite action compactification} if
    \begin{enumerate}
    \item $\cC$ is a bounded set, and
    \item zero is either a regular value or an isolated critical value of $\overline f|_{\partial \overline M}$.
    \end{enumerate}
\end{definition}

\begin{remark}
    A stronger condition would be that $\cC$ and the set of critical values of $\overline f_{\partial \overline M}$ are both finite. As we will see in \autoref{prop:finitetype}, this stronger condition holds in many examples.
\end{remark}

The next proposition follows directly from \autoref{def:finite-type}. 

\begin{proposition}
    \label{prop:finite-type}
    Let $(\overline M, \overline f)$ be a finite action compactification of $(M,f)$.
    \begin{enumerate}
        \item There exists $\varepsilon_0 > 0$ such that if $0 < |\varepsilon| \leq \varepsilon_0$, then $f_\varepsilon$ has no critical points at infinity. 
        \item Let $\lambda >0$ be such that $\cC \subset (-\lambda , \lambda )$. Then for every $\Lambda > \lambda$ there exists $\varepsilon_\Lambda \in (0,\varepsilon_0)$ such that if $|\varepsilon| \leq \varepsilon_\Lambda$, then 
        \[
            \cC_\varepsilon \subset (-\infty,-\Lambda) \sqcup (-\lambda,\lambda) \sqcup (\Lambda, +\infty).
        \]
    \end{enumerate}
\end{proposition}

The finite action condition and \autoref{prop:finite-type} imply that as $\varepsilon \to 0$ the critical \emph{values} of $f_\varepsilon$ either remain in a fixed bounded interval or diverge to infinity. The behavior of critical \emph{points} is more complicated. For a small non-zero $\varepsilon$, the set of critical points of $f_\varepsilon$ is compact by \autoref{proposition:compactness}. As $\varepsilon \to 0$, the critical points with unbounded critical values escape to infinity. As for the critical points of $f_\varepsilon$ with bounded critical values, some of them remain in $M$ and converge to critical points of $f$ whereas others escape to infinity.

In general, there is no uniform bound on the critical values of $f_\varepsilon$ as $\varepsilon\to 0$. However, in the following special case (which is, unfortunately, too restrictive for our purposes), such a bound does exist.

\begin{lemma}\label{lemma:bounded}
    Let $(\overline{M}, \overline{f} )$ be a compactification of $(M,f)$. Suppose that $\overline{M}$ is a manifold with boundary and that $\tau = \overline f / f$ vanishes transversely along $\partial \overline{M}$. Given any $\varepsilon_0 > 0$ there exists $C>0$ such that on any critical point of $f_\varepsilon$ with $|\varepsilon|\leq \varepsilon_0$ we have $| f_\varepsilon |\leq C$. 
    
    More generally, the same conclusion holds if there are a constant $c$ and a vector field $\nu$ on a neighborhood $U$ of $\partial \overline M$ such that $\partial_\nu\tau \neq 0$ on $U\setminus\partial \overline M$ and  $| \partial_\nu \overline f | \leq c | \partial_\nu \tau |$ on $M$.
\end{lemma} 

\begin{proof}
    We prove the general case. Let $p \in M$ be a critical point of $f_\varepsilon$. Suppose that $p$ belongs $U\setminus\partial\overline M$. Differentiating $\tau f_\varepsilon = \overline f + \varepsilon$ and evaluating at $p$ yields
    \[
        f_\varepsilon(p) (d \tau)_p = (d \overline f)_p.
    \]
    Applying this to $\nu(p)$ and using the estimate, we arrive at
    \[
        |f_\varepsilon(p)| |\partial_\nu \tau(p) | \leq c  |\partial_\nu \tau(p) |,
    \]
    which implies $|f_\varepsilon(p)| \leq c$. On the other hand, critical values of $f_\varepsilon$ outside of the given neighborhood of $\partial\overline M$ are bounded for small $\varepsilon$ since the critical set of $f|_{M\setminus U}$ is compact. 
    
    If $\overline M$ is a manifold with boundary and $\tau$ vanishes transversely along $\partial M$, then in a collar neighborhood $[0,1) \times \partial\overline M$ of $\partial\overline M$ we have $\tau = y$ where $y$ is the coordinate on $[0,1)$, and the vector field $\nu = \partial / \partial_y$ satifies the assumption of the general case.
\end{proof}

\begin{example}
   The following is example is non-compact but it demonstrates the general principle. Let $\overline M = [0,1)$ with
    \[
        f(y) = y \quad\text{and}\quad \tau(y) = y,
    \]
    so that $\overline f(y) = \tau(y)f(y) =y^2$ is smooth on $\overline M$. 
    The function
    \[
        f_\varepsilon(y) = f(y) + \frac{\varepsilon}{\tau(y)} = y + \varepsilon y^{-1}
    \]
    has a unique critical point $p_\varepsilon = \varepsilon^{1/2}$ in $M = (0,1)$. As $\varepsilon \to 0$, it escapes to infinity $\partial \overline M = \{ 0 \}$ but the critical values $f_\varepsilon(p_\varepsilon) = 2\varepsilon^{1/2}$ remain  bounded. 
\end{example}

\begin{example}
    If $\overline M$ is a manifold with boundary but $\tau$ vanishes to higher order at $\partial\overline M$, then some of the critical points of $f_\varepsilon$ may have unbounded critical values as $\varepsilon \to 0$. For example, for $\overline M = [0,1)$ and
    \[
        f(y) = -y^{-1}  \quad\text{and}\quad \tau(y) = \frac12 y^2,
    \]
    so that $\overline f(y) = \tau(y) f(y) = - \frac 12 y$ is smooth on $\overline M$. The perturbation
    \[
        f_\varepsilon(y) = -y^{-1} + \frac{\varepsilon}{2} y^{-2}
    \]
    has a unique critical point $p_\varepsilon = \varepsilon$ with critical value $f_\varepsilon(p_\varepsilon) = \frac12 \varepsilon^{-1}$. 
\end{example}

\begin{example}
    A sequence of critical points with unbounded critical values may also converge to a lower-dimensional stratum of $\partial\overline M$. For example, let $\overline M = [0,1)^2$ be the Euclidean corner with coordinates $(y_1,y_2)$. Consider
    \[
       f(y) = - (y_1^{-1} + y_2^{-1})  \quad\text{and}\quad \tau(y) = y_1 y_2, 
    \]
    so that $\overline f(y) = \tau(y) f(y) = -(y_1+y_2) $ is smooth on $\overline M$. The perturbation
    \[
        f_\varepsilon(y) = - (y_1^{-1} + y_2^{-1})  + \varepsilon (y_1 y_2)^{-1}
    \]
    has a unique critical point $p_\varepsilon = (\varepsilon, \varepsilon)$ with critical value $f_\varepsilon(p_\varepsilon) = -\varepsilon^{-1}$. 
\end{example}

\begin{example}
    Since $\cC$ contains the set of critical values of $f$, if the latter is unbounded, then so is $\cC$ and as a result, $(M,f)$ does not admit a finite action compactification
\end{example}

\subsection{Locally semialgebraic compactifications}
\label{sec:semialgebraic}

The finite action condition is not, in general, easy to verify. The goal of the next result is to verify it in the situation when the compactification is locally described by \emph{Nash functions}. Nash functions are generalizations of real polynomial functions, and share many of their properties. We begin by recalling their definition and basic properties; for details, see \cites{shiota, bochnak}.

Recall that a set $T\subset \R^N$ is \emph{semialgebraic} if it is a finite union of sets described by finitely many polynomial equations and inequalities, that is of the form $T = \bigcup_{i=1}^k T_i$ where each $T_i$ is of the form
\[
    T_i = \{ x \in \R^N \ | \ p_1(x) = \ldots = p_r(x) = 0, \quad q_1(x) > 0, \ \ldots , \ q_s(x) > 0 \},
\]
where $p_1, \ldots, p_r, q_1,\ldots, q_s$ are real polynomials. Given such a presentation, define the \emph{complexity} of $T_i$ by the sum of the degrees of these polynomials and the complexity of $T$ by the sum of the complexities of $T_1, \ldots, T_k$. 

We will use the following classical results about semialgebraic sets $T \subset \R^N$:
\begin{enumerate}
    \item The closure and the interior of $T$ are both semialgebraic sets. 
    \item \emph{Tarski--Seidenberg theorem.} The projection of $T$ on $\R^{N-1} \subset \R^N$ is semialgebraic. 
    \item The rank of homology $H_*(T,\Z)$ (in particular, the number of connected components) is bounded by a constant depending only the complexity of $T$ \cite[Theorem 3]{milnor}.
    \item \emph{Curve selection lemma.} For every $p_0 \in \overline T$ there exists an analytic map $p \colon [0,\varepsilon) \to \R^N$ such that $p(0) = p_0$ and $p(t) \in T$ for every $t \in (0,\varepsilon)$.
\end{enumerate}

A function $f \colon T \to \R$ is \emph{semialgebraic} if its graph 
\[
    \Gamma_f = \{ (x, f(x)) \in T \times \R \ | \ x \in T \} \subset \R^{N+1}
\]
is semialgebraic. If $T$ is a $C^k$ submanifold with corners of $\R^N$, then a \emph{$C^k$ Nash function} on $T$ is a function $f \colon T \to \R$ which is both of class $C^k$ and semialgebraic. 

\begin{definition}
    \label{definition:semialgebraic}
    A compactification $(\overline{M} , \overline{f} )$ of $(M , f )$ is \textit{locally semialgebraic} if:
    \begin{enumerate}
    \item there exists a finite open cover $\overline{U}_1 , \ldots , \overline{U}_l$ of $\overline{M}$ by $C^1$ submanifolds with corners;
    \item for each $j \in \{1 , \ldots , l\}$ there is a $C^1$ embedding $\phi_j \colon \overline U_j \hookrightarrow \R^{N_j}$ whose image $\phi_j(\overline U_j)$ is a semialgebraic $C^1$ submanifold with corners;
    \item the functions $\overline{f}\circ (\phi_j )^{-1}$ and $f \circ (\phi_j)^{-1}$ are $C^1$ Nash functions on $\phi_j(\overline U_j)$ and $\phi_j(\overline U_j \cap M)$, respectively. 
     \end{enumerate} 
\end{definition}

\begin{proposition}
    \label{prop:finitetype}
    If a compactification $(\overline M,\overline f)$ of $(M,f)$ is locally semialgebraic, then both $\cC$ and the set of critical values of $\overline f_{\partial\overline M}$ are finite. In particular, $(\overline M, \overline f)$ is a finite action compactification in the sense of \autoref{def:finite-type}. 
\end{proposition}

\begin{proof}
    After embedding, we may assume that $\overline U_j \subset \R^{N_j}$ is a $C^1$ submanifold with corners and $U_j = \overline U_j \cap M$ is its top stratum. Since $f_\varepsilon \colon U_j \to \R$ is a $C^1$ Nash function, the set of its critical points is semialgebraic \cite[Theorem 9.6.2]{bochnak}. 
    The complexity of this set depends only on the complexity of $U_j$ and the graphs of $f$ and $\overline f$; in particular, it is independent of $\varepsilon$. Since $f_\varepsilon$ is constant on each connected component of the critical locus, it follows from Milnor's theorem mentioned earlier that there exists $m_j$, independent of $\varepsilon$, such that  $f_\varepsilon$ has at most $m_j$ critical values on $U_j$ for every $\varepsilon$. Therefore, $f_\varepsilon$ has at most $m = \sum_{j = 1}^l m_j$ critical values on $M = \bigcup_{j=1}^l U_j$. 
    
    By applying this argument to the restriction of $\overline f$ to the boundary strata of $\overline{M}$ we prove that $\overline f|_{\partial\overline M}$ has finitely many critical values which verifies the second condition in \autoref{def:finite-type}.

    To verify the first condition in \autoref{def:finite-type}, we will prove that $\mathcal{C}\setminus\cC_0$ has at most $m$ elements. Suppose, by contradiction, that $\mathcal{C}\setminus\cC_0$ contains $m+1$ distinct points $c^1, \ldots, c^{m+1}$. For each $i=1,\ldots,m+1$,  let $(\varepsilon^i_n)$ be a sequence in $(0,1)$ and $(p^i_n)$ a sequence $M$ such that
    \begin{itemize}
        \item $\lim_{n\to\infty} \varepsilon^i_n = 0$,
        \item $p_i^n$ is a critical point of $f_{\varepsilon^i_n}$ with critical value $c^i_n$,
        \item $\lim_{n\to\infty} c^i_n = c^i$.
    \end{itemize}
    After passing to a subsequence $(p^i_n)$ converges to a point $p^i \in \overline M$. Choose $\delta > 0$ small enough so that any two distinct values in $\{ c^1, \ldots, c^{m+1} \}$ are distance at least $2\delta$ apart. Consider 
    \begin{align*}
        C_{j}^{c,\delta} = \{ (x,\varepsilon)\in U_j \times \mathbb{R} \ | \ \varepsilon \neq 0, \ (df_\varepsilon )_x = 0,  \  c - \delta < f_\varepsilon(x) < c+\delta \}.
    \end{align*}
    We claim that this set is semialgebraic. Indeed, by the previous argument, the set
    \[
        \{ (x,\varepsilon, y) \in U_j \times \R \times \R \ | \ (df_\varepsilon )_x = 0 \}
    \]
    is semialgebraic, and so are the sets
    \[
        \{ (x,\varepsilon,y) \in U_j \times \R \times \R \ | \varepsilon \neq 0 ,\, \ c-\delta < y < c + \delta \}
    \]
    and
    \[
        \{ (x,\varepsilon,y) \in U_j \times \R \times \R \ | y = f_\varepsilon(x) \}.
    \]
    Therefore, their intersection is semialgebraic. By the Tarski--Seidenberg theorem, so is    $C_{j}^{c,\delta}$ as the projection of the above intersection on $U_j \times \R$. 
    
    By construction, for every $i$ the point $(p^i,0) \in \overline{M} \times [0,1)$ lies in the closure of $C_{j}^{c^i,\delta}$ for some $j$. The curve selection lemma implies that there are $C^0$ curves
    $\gamma^i \colon [0,1) \to \R^{N_j}$ and $\varepsilon^i \colon [0,1) \to [0,1)$ such that $\gamma^i(0) = p^i$, $\varepsilon^i(0) = 0$, and $(\gamma^i(t), \varepsilon^i(t)) \in C_{j}^{c^i,\delta}$ for $t \in (0,1)$. In particular, $\varepsilon^i(t) > 0$ for $t>0$. Choose $\varepsilon > 0$ which is in the image of all paths $\varepsilon^i \colon [0,1) \to [0,1)$, that is $\varepsilon = \varepsilon^i(t^i)$ for some $t^i \in (0,1)$. Set $q^i = \gamma^i(t^i)$, so that $q^1, \ldots, q^{m+1}$ are critical points of $f_\varepsilon$ on $M$. The corresponding critical values are distinct because $(q^i,\varepsilon) \in C_{j}^{c^i,\delta}$ and the $c^i$'s are at least $2\delta$ apart. Therefore, $f_\varepsilon$ has at least $m+1$ critical values, which is a contradiction. This proves that $\mathcal{C} \setminus \mathcal{C}_0$ has at most $m$ elements. 
\end{proof}

\subsection{Finite action Morse complex}

Suppose that $(M,f)$ admits a finite action compactification. We will use the framework developed in \autoref{section:morsehomology} to associate to $f$ a Morse homology group which recovers the relative singular homology $H_*(X, \{ f < -\lambda \}; \Z)$ for $\lambda \gg 0$. This group will be generated by those critical points of $f_\varepsilon$, for $0 < \varepsilon \ll 1$, whose action remains finite as $\varepsilon \to 0$. 

Let $\cC$ be the set from \autoref{def:finite-type}. Choose $\lambda > 0$ large enough so that
\begin{equation}
    \label{lambda}
    \cC \subset (-\lambda , \lambda ).
\end{equation}
By \autoref{prop:finite-type}, there exists $\varepsilon_0 > 0$ such that if $0 < \varepsilon \leq \varepsilon_0$ then 
\[
f_\varepsilon ( \mathrm{Crit}f_\varepsilon ) \subset (-\infty , -2\lambda ) \sqcup (-\lambda , \lambda ) \sqcup (2\lambda , +\infty ).
\]
In particular, $\pm \lambda$ are regular values of all the functions $f_\varepsilon$ with $0 < \varepsilon \leq \varepsilon_0$. 
Moreover, by \autoref{def:finite-type}, zero is either a regular value or an isolated critical value of $\overline f|_{\partial\overline M}$. By making $\varepsilon_0$ smaller than the absolute value of the smallest negative critical value of $\overline f|_{\partial\overline M}$, we guarantee that zero is not a critical value of $(\overline f + \varepsilon) |_{\partial\overline M}$ for all $0 < \varepsilon \leq \varepsilon_0$. In other words, $f_\varepsilon$ has no critical points at infinity for all  $0 < \varepsilon \leq \varepsilon_0$.

Let $\leftindex^{b} {g}$ be a logarithmic Riemannian metric on $\overline M$ and let $g = \tau^{-1} \cdot \leftindex^{b} {g}$ be the associated conical metric, as in \autoref{subsection: palais-smale}.  By construction, for every $0 < \varepsilon \leq \varepsilon_0$, the pair $(f_\varepsilon, g)$ belongs to the parameter space $\mathscr{P}_{(-\lambda , \lambda )}$ introduced in \autoref{definition: parameters}. Thus, there is associated Morse homology $HM_*^{(-\lambda,\lambda)}(f_\varepsilon, g)$. By \autoref{proposition: continuation}, given $0 < \varepsilon' \leq \varepsilon \leq \varepsilon_0$ there is an associated continuation isomorphism associated with the map $(f_{t\varepsilon' + (1-t)\varepsilon},g)$ with $t \in [0,1]$,
\[
    HM_*^{(-\lambda,\lambda)}(f_\varepsilon, g) \xrightarrow{\cong} HM_*^{(-\lambda,\lambda)}(f_{\varepsilon'}, g).
\]
Thus, as $\varepsilon \downarrow 0$, the groups $HM_*^{(-\lambda,\lambda)}(f_\varepsilon, g) $ form a directed system. 
Of course, all maps in this directed system are isomorphisms; the reason we are giving a preference to the direction $\varepsilon \downarrow 0$ is due to the following:
\begin{remark}
More generally, suppose that $0 < \varepsilon' \leq \varepsilon$ and both pairs $(f_{\varepsilon},g)$, $(f_{\varepsilon'},g )$ are contained in $\mathscr{P}_{(-\lambda , \lambda )}$. We are no longer assuming $\varepsilon \leq \varepsilon_0$, thus there may exists $\varepsilon'' \in (\varepsilon' , \varepsilon )$ with $(f_{\varepsilon''} , g) \notin \mathscr{P}_{(-\lambda , \lambda )}$. A continuation map associated with the path $ f_{t \varepsilon' + (1-t)\varepsilon}$, $t \in [0,1]$, can still be constructed in this case,
\[
HM_*^{(-\lambda,\lambda)}(f_\varepsilon, g) \to HM_*^{(-\lambda,\lambda)}(f_{\varepsilon'}, g) .
\]
This is defined exactly as before: the key point is that $\partial_t f_{t \varepsilon' + (1-t)\varepsilon} = (\varepsilon' - \varepsilon )/\tau \leq 0$, and thus the inequality (\ref{inequality_strip}) still holds (for every $\delta >0$, in fact). However, this continuation map may no longer be an isomorphism.
\end{remark}

We have now arrived to the crucial notion in this article:

\begin{definition}
    \label{definition:finite action morse homology}
Given a finite action compactification $(\overline{M} , \overline{f} )$ of $(M, f)$, the \textit{finite action Morse homology} of $(M, f)$ is the finitely-generated graded abelian group $HM_\ast (f,g )$ defined as the direct limit 
\[
    HM_\ast (f,g ) \coloneq \varinjlim_{\varepsilon \downarrow 0} HM_\ast^{(-\lambda , \lambda )}(f_\varepsilon , g ).
\]
where $\lambda$ is any constant satisfying \eqref{lambda}. Note that, since the continuation maps are all isomorphisms for $0 < \varepsilon \leq \varepsilon_0$, we have a canonical isomorphism for each $0 < \varepsilon \leq \varepsilon_0$
\[
HM_\ast (f,g) \cong HM_\ast^{(-\lambda , \lambda)}(f_\varepsilon , g ) .
\]
Similarly, the \textit{finite action Morse cohomology} of $(M,f)$ is defined as the inverse limit
\[
    HM^\ast (f,g) \coloneq \varprojlim_{\varepsilon \downarrow 0} HM_{(-\lambda, \lambda )}^\ast (f_\varepsilon , g ).
\]
\end{definition}

\begin{remark}\label{remark:-homology}
In the above definition we have only made use of the perturbations
\[
    f_\varepsilon = f + \varepsilon/\tau \quad\text{with } \varepsilon >0.
\]
The same construction can be carried out with $\varepsilon < 0$, leading to two more groups:
\begin{align*}
HM_{\ast}^{-} (f,g) & \coloneq \varprojlim_{\varepsilon \uparrow 0} HM^{(-\lambda, \lambda )}_\ast (f_\varepsilon , g ) \\
HM^{\ast}_{-} (f,g) & \coloneq \varinjlim_{\varepsilon \uparrow 0} HM_{(-\lambda, \lambda )}^\ast (f_\varepsilon , g ).
\end{align*}
However, we shall not make much reference to these groups, as they are related via standard dualities to the previously defined groups 
\[
HM_\ast^{-} (f,g) \cong HM^{\dim M-\ast} (- f , g), , \quad HM_-^\ast (f,g) \cong HM_{\dim M-\ast} (-f , g ).
\] 
\end{remark}

\subsection{Isomorphism with singular homology}

By \autoref{proposition: continuation}, up to isomorphism the finite action Morse homology $HM_\ast (f,g)$ depends only on $f \colon M \to \mathbb{R}$ and, a priori, on the choice of the finite action compactification $(\overline{M} , \overline{f} )$. We shall now see that it is, in fact, indepedent of the choice of $(\overline{M} , \overline{f} )$. The following result and Proposition \ref{prop:finitetype} prove Theorem \ref{theorem:main_real}.

\begin{proposition}\label{proposition:isosingular}
    Let $(\overline{M} , \overline{f} )$ be a finite action compactification of $(M,f)$. For every $\lambda \gg 1$ there are canonical isomorphisms between finite action Morse (co)homology with relative singular (co)homology
    \begin{gather*}
        HM_\ast (f,g) \cong H_\ast (M , \{f < -\lambda \} ; \mathbb{Z} ), \\
        HM^\ast (f,g) \cong H^\ast ( M  , \{ f < -\lambda \} ; \mathbb{Z} ).
    \end{gather*}
\end{proposition}

\begin{proof}
    Let $(\varepsilon_n)$ be a decreasing sequence with $0 < \varepsilon_n  \leq \varepsilon_0$ and $\varepsilon_n \to 0$. By \autoref{prop:finite-type}, there exists an unbounded increasing sequence $\Lambda_n  > \lambda$ such that
    \[
    f_{\varepsilon_n}(\mathrm{Crit}f_{\varepsilon_n}) \subset (-\infty , -\Lambda_n ) \sqcup (-\lambda , \lambda ) \sqcup (\Lambda_n , +\infty ).
    \]
    It follows from the definition of Morse homology that for every $n$,
    \[
        HM_\ast^{(-\lambda , \lambda )}(f_{\varepsilon_n} , g ) \cong HM_\ast^{(-\lambda , \Lambda_n )}(f_{\varepsilon_n} , g ).
    \]
    We have the following string of isomorphisms:
    \begin{align*}
        HM_\ast (f,g) & = \varinjlim_{\varepsilon \downarrow 0} HM_\ast^{(-\lambda , \lambda )}(f_\varepsilon ,g)\\
        & \cong \varinjlim_{n \to \infty} HM_\ast^{(-\lambda , \lambda )}(f_{\varepsilon_n},g)\\
        & \cong \varinjlim_{n \to \infty} HM_\ast^{(-\lambda , \Lambda_n )}(f_{\varepsilon_n},g)\\
        & \cong \varinjlim_{n \to \infty}H_\ast (\{ f_{\varepsilon_n} \leq \Lambda_n \} , \{f_{\varepsilon_n} \leq -\lambda \} ; \mathbb{Z} ) \quad \text{by \autoref{proposition:morse homology}}\\
        & \cong H_\ast ( M , \{ f< -\lambda \} ; \mathbb{Z} ).
    \end{align*}
    Here, the last isomorphism follows from the fact that the inclusions 
    \[
        \{f_{\varepsilon_n} \leq \Lambda_n\} \subset \{f_{\varepsilon_{n+1}}\leq \Lambda_{n+1}\} \quad\text{and}\quad \{f_{\varepsilon_n} \leq -\lambda\} \subset \{f_{\varepsilon_{n+1}}\leq -\lambda\}
    \] 
    are cofibrations, and that $M$ and $\{ f < -\lambda \}$ are exhausted by the subspaces $\{ f_{\varepsilon_n}\leq \Lambda_n \}$ and $\{ f_{\varepsilon_n} \leq - \lambda \}$, respectively, as $n\to\infty$. Thus standard convergence theorems for singular homology apply; see \cite[Proposition 10.8.1 and Proposition 10.8.4]{tomdieck}.
    
    As for cohomology, we similarly obtain
    \begin{align*}
        HM^\ast (f,g) & \cong \varprojlim_{n \to \infty} HM^\ast_{(-\lambda , \Lambda_n )}(f_{\varepsilon_n},g)\\
        & \cong \varprojlim_{n \to \infty}H^\ast (\{ f_{\varepsilon_n} \leq \Lambda_n \} , \{f_{\varepsilon_n} \leq -\lambda \} ; \mathbb{Z} ) \quad \text{by \autoref{proposition:morse homology}}\\
        & \cong H^\ast ( M , \{ f< -\lambda \} ; \mathbb{Z} )
    \end{align*}
    where in order to deduce the last isomorphism we now make use of the fact that the limit as $n \to \infty$ of the groups $H^\ast (\{ f_{\varepsilon_n} \leq \Lambda_n \} , \{f_{\varepsilon_n} \leq -\lambda \} ; \mathbb{Z} )$ is stabilizing (which follows from the fact that the corresponding limit of Morse cohomologies  stabilizes),  together with standard convergence results for singular cohomology (\cite[Proposition 10.8.4, Proposition 17.1.6 and Proposition 17.1.7]{tomdieck}).
\end{proof}

\begin{remark}
    Using standard dualities in Morse homology and singular (co)homology, we deduce from \autoref{remark:-homology} and  \autoref{proposition:isosingular} the following canonical isomorphisms:
    \begin{align*}
    HM^-_\ast (f,g) & \cong H^{n-\ast} (M , \{f > \lambda \})\cong H_{\ast}^{BM}(M , \{f>\lambda \}),\\
    HM_-^\ast (f,g) & \cong H_{n-\ast}(M , \{ f>\lambda \}) \cong H^{\ast}_c(M, \{f>\lambda \}).
\end{align*}
\end{remark}

\section{Morse homology of complex algebraic varieties}
\label{section:algebraic}

The goal of this section is to associate a Morse homology group $HM_*(X,F)$ with a regular function $F \colon X \to \C$ on a smooth complex algebraic variety, and to prove \autoref{theorem:main1}. The construction proceeds as follows. Using Hironaka's resolution of singularities and the real blow-up construction, we build a compactification $(\overline M, \overline F)$ of $(X,F)$ in the sense of \autoref{section:compactifications}. Then $HM_*(X,F)$ is defined as the finite action Morse homology, introduced in \autoref{section:finiteaction}, of the real part $\Re(F)$. 

\subsection{Almost complex compactifications}

\autoref{definition:compactification} of the compactification can be adapted to take into account almost complex structures. Recall that an almost complex structure $I$ on a manifold $M$ is endomorphism $I : TX \to TX$ satisfying $I^2 = -1$.  

\begin{definition}
~
\begin{enumerate}
    \item Let $\overline{X}$ be a manifold with corners. A \textit{logarithmic almost complex structure} on $\overline{X}$ is a smooth bundle endomorphism $\leftindex^{b}I$ of $\leftindex^{b} {T}\overline{X}$ satisfying $(\leftindex^{b}I)^2 = -1$. 
    \item A \textit{logarithmic almost Hermitian structure} on $\overline{X}$ is a pair $(\leftindex^{b}I , \leftindex^{b} {g})$ where $\leftindex^{b}I$ is a logarithmic almost complex structure on $\overline{X}$ and $\leftindex^{b} {g}$ is a logarithmic Riemannian metric on $\overline{X}$ such that $\leftindex^{b}I$ an isometry with respect to $\leftindex^{b}g$.
\end{enumerate}
\end{definition}

Given a logarithmic almost complex structure $\leftindex^{b}I$ on $\overline{X}$ we can always find a logarithmic metric $\leftindex^{b} {g}$ which makes $(\leftindex^{b}I , \leftindex^{b} {g})$ into a logarithmic almost Hermitian structure on $\overline{X}$. In fact, the space of all such $\leftindex^{b} {g}$ is contractible. 

\begin{definition}\label{definition:complexcompactification}  
    Let $(X,I)$ be an almost complex $2n$--manifold and $F : X \to \mathbb{C}$ an $I$--holomorphic function. An \emph{almost complex compactification} of $(X,I,F)$ is a pair $(\overline X, \overline F)$ consisting of a compact manifold with corners $\overline X$ and a smooth function $\overline F \colon \overline X \to \C$ such that
    \begin{enumerate}
        \item $X$ is diffeomorphic to the interior of $\overline X$, 
        \item $\overline F / F$ extends to a boundary function $\tau: \overline{X} \to \mathbb{R}$ in the sense of \autoref{definition:compactification}, 
        \item $I$ extends to a logarithmic almost complex structure on $\overline{X}$. 
    \end{enumerate}
\end{definition}

\begin{remark}
    One should not confuse $\overline F$ with the complex conjugation of $F$, which is never used in this paper. 
\end{remark}

Let $(\overline X, \overline F)$ be an almost complex compactification of $(X,I,F)$. For every $\theta \in \R/2\pi\Z$ define
\[
    f_\theta \coloneq \Re(e^{i\theta}F) \quad\text{and}\quad \overline f_\theta \coloneq \Re(e^{i\theta}\overline F). 
\]
The real pair $(\overline X, \overline f_\theta)$ is a compactification of $(X,f_\theta)$ in the sense of \autoref{definition:compactification}. Since $F$ is $I$--holomorphic,
\[
    \mathrm{Crit}(F) = \mathrm{Crit}(f_\theta) \quad\text{ for all } \theta \in \R/2\pi \Z. 
\]
The following result allows us to control critical points at infinity for all $\theta$. As before, consider the perturbation
\[
    f_{\theta,\varepsilon} \coloneq f_\theta + \frac{\varepsilon}{\tau} = \Re(e^{i\theta} F) + \frac{\varepsilon}{\tau}.
\]




\begin{proposition}\label{proposition:critholo}
    Let $(\overline{X}, \overline{F} )$ be an almost complex compactification of $(X, I , F)$. Suppose that $0$ is either a regular value or an isolated critical value of $| \overline{F}|^2 : \partial \overline{M} \to \mathbb{R}$, in the stratified sense. Then there exists $\varepsilon_0 >0$ such that the function $f_{\theta,\varepsilon}$ has no critical points at infinity for all $\varepsilon$ with $0 < |\varepsilon|\leq \varepsilon_0$ and all $\theta \in \mathbb{R}/2\pi \mathbb{Z}$.
\end{proposition}

\begin{proof}
    We have the identity of logarithmic $1$-forms $ \tau dF = d \overline{F} - \overline{F} d\tau /\tau$. Let $(I , \leftindex^{b} {g} )$ be a logarithmic Hermitian almost complex structure on $\overline{X}$. Let $g= \tau^{-1} \leftindex^{b} {g}$ denote the associated cone metric on $X$. Consider the logarithmic vector fields $\nabla (1/\tau)$ and $R := I \nabla ( 1/\tau )$, where $\nabla ( 1/\tau )$ is the gradient of $1/\tau$ for the metric $g$.  Evaluating the identity at points on $\partial \overline{X}$ against these two vector fields, then using the fact that they are $\leftindex^{b} {g}$--orthogonal, and
    the fact that $d\overline{F} ( \nabla (1/\tau )) = 0$ on $\partial \overline{X}$ because $\nabla (1/\tau )$ is a section of the logarithmic normal sheaf of $\partial \overline{X}$, yields:
    \begin{align*}
    \tau \dF ( \nabla \frac{1}{\tau}) & =  \overline{F} \cdot \leftindex^{b}{| \frac{d\tau}{\tau}|}^{2}\\
    \tau dF ( R) & = ( d \overline{F} )(R).
    \end{align*}
   Since $F$ is holomorphic in the complex structure $I$, it follows that we have the following identity at points of $\partial \overline{X}$:
    \begin{align}
    d\overline{F} (R) = i \overline{F} \cdot \leftindex^{b}{| \frac{d\tau}{\tau}|}^{2} ,\label{holo_infinity}
    \end{align}
    and note that the quantity $\leftindex^{b}{| \frac{d\tau}{\tau}|}^{2}$ is non-vanishing along $\partial \overline{X}$ by \eqref{dtau/tau}.
    
    Now, suppose that $p \in \partial \overline{X}$ is a critical point at infinity of $\mathrm{Re}(e^{-i \theta} F) + \varepsilon/\tau$. We will show that $\varepsilon^2$ is a stratified critical value of $|\overline{F}|^2 : \partial \overline{X} \to \mathbb{R}$. From the hypothesis that if $0$ is a stratified critical value for this function then it is an isolated one, the result will then follow. To show this, we may replace $F$ by $e^{i\theta}F$ so as to suppose that $\theta = 0$. Thus, $p$ is a critical point of $\mathrm{Re}\overline{F}$ along the stratum $S \subset \partial \overline{X}$ containing $p$ and with critical value $- \varepsilon$. By \eqref{holo_infinity} we obtain
    \[
    0 =  (d \mathrm{Re} \overline{F})_p (R) = - i \mathrm{Im}\overline{F} (p) \cdot \leftindex^{b}{| \frac{d\tau}{\tau}|}^{2} 
    \]
    and thus $\mathrm{Im} \overline{F}(p) = 0$. Thus, $|\overline{F}(p) |^2 = \varepsilon^2 $ and along the stratum $S \subset \partial \overline{X}$ containing $p$ we have
    \[
    (d |\overline{F}|^{2} )_p = 2 \mathrm{Re}\overline{F}(p) (d\mathrm{Re}\overline{F})_p + 2 \mathrm{Im}\overline{F}(p) (d\mathrm{Im}\overline{F})_p = 0.
    \]
    Thus $\varepsilon^2$ is a stratified critical value of $|\overline{F}|^2 : \partial \overline{X} \to \mathbb{R}$, and the result follows.
\end{proof}

\subsection{Compactifying complex algebraic varieties}

Let $X$ be a smooth complex algebraic variety, with complex structure denoted $I$. By Nagata's compactification theorem \cite{nagata62,nagata63} and Hironaka's resolution of singularities \cite{hironaka1,hironaka2}, there exist
\begin{enumerate}
    \item a smooth, complete complex algebraic variety $Y$, and \
    \item a simple normal crossing divisor $D = D_1 + \ldots + D_\ell$ in $Y$,
\end{enumerate}
such that $X$ is isomorphic to $Y\setminus D$ as complex algebraic varieties. The completeness condition is equivalent to $Y$ being compact with respect to the analytic topology. Recall also that the simple normal crossing condition means that each irreducible component $D_i$ is smooth and that around every point of $D$ there exist analytic coordinates $(z_1, \ldots, z_m)$ on $Y$, where $m = \dim_\C(Y)$, in which $D = \{ z_1 \cdots z_\ell = 0 \}$. Globally, let $\pi_i \colon L_i \to Y$ be an algebraic line bundle with a section $s_i \in H^0(Y,L_i)$ vanishing transversely along $D_i$, so that $s_i$ defines an isomorphism $L_i \cong \sO_Y(D_i)$. The divisor $D$ is the zero set of 
\[
    s \coloneq \bigotimes_{i=1}^\ell s_i \in H^0(Y,L) \quad\text{where } L \coloneq \bigotimes_{i=1}^\ell L_i.
\]

\begin{definition}
~
\begin{enumerate}
\item
    Let $D = s^{-1}(0)$ be an effective divisor, where $s \in H^0(Y,L)$ is a non-zero section of a line bundle $\pi \colon L \to Y$. The \emph{simple real blowup} of $Y$ along $D$ is the quotient
    \[
        \mathrm{sBl}_D(Y) \coloneq \frac{ \{ v \in L\setminus\{0\} \ |  s(\pi(v)) \in \R_{\geq 0} v \}     }{\R_{>0}}.
    \]
    where $\R_{>0}$ acts on $L$ by rescaling. The map $\pi$ induces a projection $\mathrm{sBl}_D(Y) \to Y$.
    \item The \emph{real blowup} of $Y$ along a simple normal crossing divisor $D = D_1 + \ldots + D_\ell$ is the fiber product
    \[ 
        \mathrm{Bl}_D(Y) \coloneq \mathrm{sBl}_{D_1}(Y) \times_Y \mathrm{sBl}_{D_2}(X)  \times_Y \cdots \times_Y \mathrm{sBl}_{D_\ell}(Y) .
    \]
\end{enumerate}
\end{definition}

\begin{remark}
    When $D$ is smooth, the simple blowup $\mathrm{sBl}_D(Y)$ is a manifold with boundary. In general, $\mathrm{sBl}_D(Y)$ is only a Whitney stratified space. The advantage of working with $\mathrm{Bl}_D(Y)$ is that, as we will see, it is a manifold with corners when $D$ is a simple normal crossing divisor. In fact, $\mathrm{Bl}_D(Y) \to \mathrm{sBl}_D(Y)$ is the canonical resolution of a stratified space by a manifold with corners in the sense of \cite[Section 2]{albin}.
\end{remark}

We will consider the compactification of $X$ given by
\[
    \overline X \coloneq \mathrm{Bl}_D(Y). 
\]
As we will see, $\overline X$ is a compact manifold with corners, containing $X$ as the smooth stratum. To define a compactification of the \emph{pair} $(X,F)$, where $F :X \to \mathbb{C}$ is a regular function, we will construct a boundary function $\tau \colon \overline X \to [0,1]$ such that $\overline F = \tau F$ extends to a smooth function on $\overline X$. 

First, observe that $F$ extends to a rational function on $Y$. Denote by $\mathrm{ord}(F,D_i)$ the order of $F$ along $D_i$, so that $\mathrm{ord}(F,D_i) < 0$ corresponds to $F$ having a pole along $D_i$ and $\mathrm{ord}(F,D_i) > 0$ corresponds to $F$ vanishing along $D_i$. Let $\alpha = (\alpha_1, \ldots, \alpha_\ell)$ be any $\ell$--tuple of positive integers satisfying
\begin{equation}
    \label{divisor weights}
    \alpha_i \geq \mathrm{max}(-\mathrm{ord}(F,D_i ) , 1 ).
\end{equation}
Set
\[
    L^\alpha = \bigotimes_{i=1}^\ell L_i^{\alpha_i}  \quad\text{and}\quad s^\alpha = \bigotimes_{i=1}^\ell s_i^{\alpha_i} \in H^0(Y, L^\alpha).  
\]
The significance of this construction is that  $F: X \to \mathbb{C}$ can be expressed as
\begin{equation}
    \label{F as a quotient}
    F = \frac{t}{s^\alpha} 
\end{equation}
for a section $t \in H^0 (Y, L^\alpha )$. Thus, to extend $F$ to a smooth function on $\overline X$ we need to multiply it by a real function which vanishes at the same rate as $s^\alpha$ as we approach each $D_i$. A natural candidate is the norm of $s^\alpha$ with respect to a Hermitian metric on $L^\alpha$, but this has to be done with some care in order to preserve the algebraic nature of the compactification.

Ideally, we would use a real algebraic Hermitian metric, but we do not know if every algebraic line bundle over a smooth complex algebraic variety admits such an inner product (for the projective case, see \autoref{remark:metricprojective} below). However, since $Y$ is a complex algebraic variety, it is naturally a \emph{$C^k$ Nash manifold}, that is: a manifold covered by charts homeomorphic to $C^k$ semialgebraic submanifolds of $\R^N$ with $C^k$ Nash transition functions. Similarly, $L_i \to Y$ is a \emph{$C^k$ Nash vector bundle}. 

\begin{lemma}
    Every $C^k$ Nash complex line bundle over a $C^k$ Nash manifold admits a $C^k$ Nash Hermitian metric. 
\end{lemma}
\begin{proof}
    Let $L \to Y$ be a $C^k$ Nash complex line bundle. A generalization of the Nash embedding theorem asserts that every $C^k$ Nash manifolds is isomorphic to an affine one; in fact, to a smooth  affine real algebraic variety \cite[Theorem III.1.1]{shiota}. Thus, we may assume that both $L$ and $Y$ are $C^k$ affine Nash manifolds, and so is the bundle $H \to Y$ of Hermitian pairings $L \otimes_\R L \to \C$. Let $h \colon Y \to H$ be a $C^k$ section corresponding to a Hermitian metric on $L_i$.  In the affine setting, any $C^k$ section can be approximated in the $C^k$ topology by a $C^k$ Nash section \cite[Remark III.2.5]{shiota}. Therefore, there exists a $C^k$ Nash section $\tilde h$ of $H$ which is close to $h$, and so non-degenerate. 
\end{proof}

\begin{remark}\label{remark:metricprojective}
If $Y$ is \textit{projective}, then  $L_i \to Y$ admits a real algebraic Hermitian metric. Indeed, fixing a projective embedding $e: Y \hookrightarrow \mathbb{P}^N$ yields a canonical line bundle $\mathscr{O}_Y (1) = e^\ast \mathscr{O}_{\mathbb{P}^N } (1)$, and define also $\mathscr{O}_Y (n) := \mathscr{O}_Y (1)^{\otimes n }$ for $n \in \mathbb{Z}$. The line bundle $\mathscr{O}_{\mathbb{P}^N} (-1)$ embeds into the trivial bundle of rank $N+1$, hence inherits a real algebraic Hermitian metric by restriction; hence all line bundles $\mathscr{O}_Y (n)$ carry such a metric. By \cite[Corollary 5.18]{hartshorne}, any complex algebraic vector bundle over $Y$ can be embedded as a complex algebraic subbundle of a sum of line bundles of the form $\mathscr{O}_Y (n)$, which yields a real algebraic Hermitian metric.
\end{remark}

Assume from now on that each $L_i$ is equipped with a $C^k$ Nash Hermitian metric, denoted by $\langle \cdot, \cdot \rangle_i$, and define  $\tau \colon \overline M \to [0,\infty)$ by
\begin{equation}
    \label{tau definition}
    \tau(v) \coloneq \prod_{i=1}^\ell \langle s_i(\pi_i(v)), v \rangle_i^{\alpha_i}.
\end{equation}
Finally, set
\[
    \overline F \coloneq \tau F.
\]

\begin{proposition}\label{proposition:SNC}
    The pair $(\overline X, \overline F)$ is a locally semialgebraic compactification of $(X,F)$ in the sense of \autoref{definition:semialgebraic}, and an almost complex compactification in the sense of \autoref{definition:complexcompactification}.
\end{proposition}

\begin{remark}
An explicit conical metric on $X$ (for the compactification $\overline{X}$ above described) and compatible with the complex structure $I$ can be obtained as follows: choosing any Riemannian metric $g_Y$ on $Y$ compatible with the complex structure, set
\[
g := \frac{1}{\tau}\Big(  \big(\frac{d\tau}{\tau}\big)^2 + \big( \frac{d\tau\circ I}{\tau}\big)^2 + g_Y \Big).
\]
Moreover, $g$ is Kähler if $g_Y$ is Kähler.
\end{remark}

\begin{proof}
    First, we will show that every point in $\overline X$ has a neighborhood homeomorphic to a semialgebraic set. Let $A \subset Y$ be a Zariski open subset isomorphic to a smooth affine complex variety $A \subset \C^N$ such that the restriction of each $L_i$ to $A$ is algebraically isomorphic to $A \times \C$. Therefore, the intersection of $\overline X$ with $\pi^{-1}(A)$ is homeomorphic to
    \[
        \overline X \cap \pi^{-1}(A) \cong \{ (x,v_1,\ldots, v_\ell) \in A \times \C^\ell \ | \ \langle s_i(x), v_i \rangle_i \geq 0 \text{ for } i=1,\ldots,\ell\},
    \]
    where $s_i \colon A \to \C$ is a real algebraic function and $\langle \cdot, \cdot \rangle_i$ is the chosen $C^k$ Nash Hermitian metric on $L_i$. It follows that $\overline X \cap \pi^{-1}(A)$ is homeomorphic to a semialgebraic set, and under this homeomorphism, $F$ is a $C^\infty$ Nash function on $X \cap \pi^{-1}(A)$ by \eqref{F as a quotient}, and $\tau$ is a $C^k$ Nash function on $\overline X \cap \pi^{-1}(A)$ by \eqref{tau definition}. 

    Next, we show that $\overline X \cap \pi^{-1}(A)$---identified with the semialgebraic set above---is a real $C^k$ submanifold with corners of $\C^{N+\ell}$, that $\tau$ is a boundary function, and $\overline{F} = \tau F$ extends to a $C^k$ function on $\overline X \cap \pi^{-1}(A)$. Note that all of these properties are invariant under $C^k$ diffeomorphisms, so at this point we may use any $C^k$ charts, not necessarily semialgebraic. The charts we will construct will automatically endow $\overline X$ with the structure of a manifold with corners such that the restriction of $\pi$,
    \[
        \pi \colon \overline X \setminus \pi^{-1}(D) \to Y\setminus D \cong X
    \]
    is a diffeomorphism, and that the complex structure on $M$ extends to a logarithmic complex structure on $\overline{X}$. 
 
    Let $p \in A$ be a point where $k \leq \ell$ divisors intersect (and $k \leq m = \mathrm{dim}_\mathbb{C} Y$ by the simple normal crossing condition); without loss of generality, $D_1, \ldots, D_k$. Let $U$ be a small open neighborhood of $p$ in the analytic topology, and let $z_1, \ldots, z_m$ be holomorphic coordinates on $U$ centered at $p$ and such that $D_i = \{ z_i = 0 \}$. After picking smooth Hermitian metrics and local trivializations of $L_i$ over $U$< 
    \[
        \pi^{-1}(U) = \{ (z,w) \in \C^m \times (S^1)^k \ | \ z_i w_i^{-1} \geq 0 \text{ for } i =1,\ldots, k\}.
    \]
    Define 
    \begin{gather*}
        \phi_U \colon \pi^{-1}(U) \to \R^k_{\geq 0} \times (S^1)^k \times \C^{m-k} \\
        \phi_U(z,w) = (z_1 w_1^{-1}, \ldots, z_k w_k^{-1}, w_1, \ldots, w_k, z_{k+1}, \ldots, z_m).
    \end{gather*}
    The collection of corner charts $(U,\phi_U)$ defines a corner atlas on $\overline X$.

    To check that $\overline f$ and $\tau = \overline f / f$ extend to $\overline X$, let $z = (z_1,\ldots, z_m)$ be local holomorphic coordinates on $X$. Set $u = (z_{k+1}, \ldots, z_m)$ and let $(r,w,u)$ be coordinates on $\R^k_{\geq 0} \times (S^1)^k \times \C^{m-k}$. We have $z(r,w,u) = (r_1w_1, \ldots, r_k w_k, u)$.  
    In these coordinates, sections $s_i$ and $s^\alpha$ introduced before \eqref{F as a quotient} are given by
    \[
        s_i(z) = z_i \quad\text{and}\quad s^\alpha(z) =  \psi (z) \prod_{i=1}^k z_i^{\alpha_i},
    \]
    where $\psi (z)$ is a nowhere vanishing holomorphic function. Similarly, section $t$ defined by \eqref{F as a quotient} is identified with a holomorphic function of $z$. We have
    \[
        \tau(r,w,u) = \varphi(z) \prod_{i=1}^k r_i^{\alpha_i}. 
    \]
    where $\varphi$ is a nowhere vanishing smooth function. Therefore,
    \[
        \overline F(z) = \tau(z) F(z) =  \frac{\varphi(z) t(z)}{\psi(z)}\prod_{i=1}^k w_i^{-\alpha_i}.
    \]
    Therefore, $\tau$ is a boundary function on $\overline X$ and $\overline F$ extends to a smooth function on $\overline X$. 
    
    Let $I$ be the given complex structure on $X$, regarded as an endomorphism of $TX$ with $I^2 = - 1$. Consider the coordinates $(r, w , u )$ on the corner chart $\mathbb{R}_{\geq 0}^k \times (S^1 )^k \times \mathbb{C}^{m-k}$ of $\overline{X}$ from the previous paragraph, and write $u_j = a_j +i b_j $ for real coordinates $a_j, b_j$. Then in $X = \mathrm{Int}\overline{X}$ we have
    \[
    I \left( \frac{\partial}{\partial w_i } \right) = - r_i \frac{\partial}{\partial r_i } \quad , \quad I \left( \frac{\partial}{\partial a_j} \right) = \frac{\partial}{\partial b_j}.
    \]
    From this, it follows immediately that $I$ extends to a logarithmic complex structure on $\overline{X}$. The proof is now complete.
\end{proof}

We can now prove the main theorem of this article.

\begin{proof}[{\bfseries Proof of \autoref{theorem:main1}}]
    Let $X \colon F \to \C$ be a regular function on a complex algebraic variety. Let $(\overline X, \overline F)$ be the compactification constructed above, for any choice of $\alpha = (\alpha_1,\ldots,\alpha_\ell)$ satisfying condition \eqref{divisor weights}. By \autoref{proposition:SNC}, $(\overline X, \overline F)$ is a locally semialgebraic compactification, and an almost complex compactification. Let $f = \Re(F)$ and consider the finite action Morse homology 
    \[
        HM_*(X,F) \coloneq HM_*(f,g)
    \]
    introduced in \autoref{definition:finite action morse homology}, for any choice of a conical metric $g$ on $X$ as in \autoref{definition: parameters}. By \autoref{theorem:main_real}, this group is well-defined and canonically isomorphic to relative singular homology
    \[
        HM_*(X,F) \cong H_*(X, \Re(F) < -\lambda) \quad\text{for } \lambda \gg 0. \qedhere
    \]
\end{proof}

\subsection{Remarks on monodromy}
In the situation of \autoref{theorem:main1}, it is natural to ask whether we can assign a \textit{monodromy} isomorphism $HM_*(X,F) \to HM_*(X,F)$ to the loop of functions $\mathrm{Re}(e^{-i\theta}F)$, $\theta \in S^1$, by using a continuation map associated to the functions
\[
f_{\theta,\varepsilon} := \mathrm{Re}(e^{-i\theta}F) + \varepsilon/\tau .
\]
A first step towards this is \autoref{proposition:critholo} which asserts that flr $0 < \varepsilon \ll1$ the functions $f_{\theta, \varepsilon}$ have no critical points at infinity for all $\theta \in S^1$. However, we do not know if for each $\theta$ the constants $\lambda (\theta )$ such that $\mathcal{C}(f_\theta ) \in (-\lambda(\theta) , \lambda(\theta) )$ (as in \autoref{def:finite-type}) have a uniform upper bound independent of $\theta$. If they did, then the continuation map of $f_{\theta, \varepsilon}$ is defined, giving an isomorphism
\begin{align}
m : HM_\ast (X, F) \xrightarrow{\cong} HM_\ast (X, F) \label{monodromy}
\end{align}
and the finite action Morse homology groups $HM_\ast (X, e^{-i\theta}F )$ thus form a local system over $S^1$. We conjecture this is the case:

\begin{conjecture}\label{conjecture}
Let $F : X \to \mathbb{C}$ be a regular function on a complex algebraic variety $X$. If $\tau : \overline{X} \to \mathbb{R}$ denotes a boundary function associated to a simple normal crossings compactification of $X$, as above constructed, then there exists $\lambda  = \lambda(F)>0$ such that: if $x_n \in X$, $\varepsilon_n \in [0,1]$ and $\theta_n \in S^1$ are sequences with $f_{\theta_n , \varepsilon_n} (x_n ) \to c$, $\theta_n \to \theta_0$ and $\varepsilon \to 0$, then $c \in  (-\lambda , \lambda )$.
\end{conjecture}

Indeed, by Verdier's result \cite[Corollaire 5.1]{verdier} (discussed in \autoref{introduction:vanishing}) one sees that the corresponding homologies of vanishing cycles $H_\ast (X, e^{-i\theta}F)$ form a local system over $S^1$, and from the canonical isomorphism $H_\ast (X, e^{-i\theta}F ) \cong H_\ast (X,e^{-i\theta}F)$ then the finite action Morse homologies form a local system as well---\autoref{conjecture} would allow for a Morse-theoretic construction of this local system structure.

In the special situation when $D \subset Y$ is a smooth divisor and $F$ is such that one can choose $\alpha = \alpha_1 = 1$, then  \autoref{lemma:bounded} applies, and in particular implies  \autoref{conjecture} in this case. For example, this includes the clasical case studied in Picard--Lefschetz theory, where one chooses a very ample line bundle $L\to Y$, two holomorphic sections $s,t \in H^0 (Y,L)$ such that $D = s^{-1}(0)$ is smooth and $s,t$ generate a Lefschetz pencil of hypersurfaces in $Y$, and $F := t/s : X = Y \setminus D \to \mathbb{C}$. The monodromy map \eqref{monodromy} in this case has a simple interpretation in terms of intersection numbers of vanishing cycles of $F$.

\appendix
\section{Perverse sheaf of vanishing cycles}

In this section, we review the construction of the sheaf of vanishing cycles of a holomorphic function, and discuss some properties of it that we will use. We then prove Corollary \ref{corollary:sheaf}.

For a (reduced) complex-analytic space $X$, we denote by $D^{b}_c (X)$ the bounded derived category of constructible sheaves of abelian groups on $X$. A typical object of this category (a complex of sheaves) will be denoted by $A^\bullet \in D_{c}^b (X)$. We refer to \cite{dimca} for the basic definitions and results about the derived category and derived functors.

Let $F : X \to \mathbb{C}$ be a complex-analytic function on the complex-analytic space $X$. We denote the fibers of $F$ by $X_t = F^{-1}(t)$, and the fiber inclusion by $\iota_t : X_t \hookrightarrow X$. In \cite{deligne}, Deligne introduced the functors of nearby and vanishing cycles of $F$ (at the value $0 \in \mathbb{C}$):
\[
\psi_F \, , \phi_F \,  : D_{c}^b (X) \to D_{c}^b (X_0 ) .\\
\]
To define these, we consider the following diagram of maps:
\[
\begin{tikzcd}
X & \arrow{l}{j} X \setminus X_0 \arrow{d}{F} & \arrow{l}{\pi} E \arrow{d} \\
X_0 \arrow{u}{\iota_0} & \mathbb{C}^\ast & \arrow{l}{\mathrm{exp}} \mathbb{C} .
\end{tikzcd}
\]
Here $E$ is the analytic space defined so that the right-most square is a pullback square. 
The functor of \textit{nearby cycles} of $F$ (at the value $0 \in \mathbb{C}$) is defined by
\[
\psi_F (A^\bullet ) := \iota_{0}^\ast \circ R(j \circ \pi )_\ast \circ (j \circ \pi )^\ast A^\bullet .
\]
We have a natural transformation $\mathrm{Id} \to R(j\circ \pi )_\ast \circ (j \circ \pi )^\ast $ arising from adjunction (namely, $(j\circ \pi)^\ast$ is left-adjoint to the functor $R (j\circ \pi)_\ast$), which induces a morphism
\[
\iota_{0}^\ast A^\bullet \to \psi_F (A^\bullet ),
\]
and the functor of \textit{vanishing cycles} of $F$ (at the value $0 \in \mathbb{C}$) is defined as the cone of this morphism:
\[
\phi_F (A^\bullet ) := Cone ( \iota_{0}^\ast A^\bullet \to \psi_F (A^\bullet ) ).
\]
Thus, we have an exact triangle in $D_{c}^b (X_0)$ for all $A^\bullet \in D_{c}^b (X)$:
\begin{align}
\iota_{0}^\ast A^\bullet \to \psi_F A^\bullet \to \phi_F A^\bullet \xrightarrow{[+1]} \label{trianglepsiphi}
\end{align}

We will only be interested in the value of these functors on $A^\bullet =\underline{\mathbb{Z}}$, the constant sheaf on $X$ supported in degree zero. The stalks of the cohomology sheaves of $\psi_{F} \underline{\mathbb{Z}}$ (resp. $\phi_F \underline{\mathbb{Z}}$) at a point $ x \in X_0$ recover the integral $\mathbb{Z}$-cohomology of the \textit{Milnor fiber} of $F$ at $x$ (resp. the \textit{local Morse group} at $x$), see \cite[\S 4.2]{dimca}:
\begin{align*}
& \mathscr{H}^\ast ( \psi_F \underline{\mathbb{Z}} )_x \cong H^\ast ( B_{\varepsilon} (x ) \cap X_t , \underline{\mathbb{Z}} )\\
& \mathscr{H}^{\ast} ( \phi_F \underline{\mathbb{Z}} )_x \cong H^{\ast+1} ( B_{\varepsilon}(x) , B_{\varepsilon} (x ) \cap X_t   ; \mathbb{Z} ),
\end{align*}
and the exact triangle (\ref{trianglepsiphi}) reduces, upon passing to stalk cohomologies, to the long exact sequence in cohomology of the pair $(B_\varepsilon (x) , B_\varepsilon (x) \cap X_t )$. 
Here $B_{\varepsilon}(x) \subset X$ is obtained by intersecting $X$ with a small radius ball around $0$ in $\mathbb{C}^N$ for a given local analytic embedding $(X, x) \subset (\mathbb{C}^N , 0 )$, and $0 <|t | \ll \varepsilon$. In particular, when $X$ is \textit{non-singular} and $x$ is a regular point of $F$ then $\mathscr{H}^{\ast} ( \phi_F \underline{\mathbb{Z}} )  $ vanishes. Hence the support of $\phi_F \underline{\mathbb{Z}}$ is contained in the critical locus $\mathrm{Crit}F $ of $F : X \to \mathbb{C}$. 

Let $C(F)= F ( \mathrm{Crit}F) \subset \mathbb{C}$ denote the subset of critical values of $F$. Taking into account all of these leads to the following:

\begin{definition}[\cite{Brav2015,bussi}] Suppose that $X$ is non-singular. The \textit{sheaf of vanishing cycles} of $F$ is the object $\mathscr{P}_{F}^\bullet \in D_{c}^b (\mathrm{Crit}F)$ 
given by
\[
\mathscr{P}_{F}^\bullet = \bigoplus_{ z \in C(F)}  ( \phi_{F-z}  \underline{\mathbb{Z}} )|_{\mathrm{Crit}F \cap X_z} .
\]
\end{definition}

From now on, we shall focus on the case when $X$ is a non-singular algebraic variety over $\mathbb{C}$, and $F : X \to \mathbb{C}$ is a regular function. By \cite[Corollaire 5.1]{verdier}, there exists a minimal \textit{finite} set $B(F) \subset \mathbb{C}$ such that
\[
F: F^{-1} ( \mathbb{C} \setminus B(F) ) \to \mathbb{C} \setminus B(F) 
\]
is a $C^\infty$ locally trivial fibration. 

\begin{definition}
A value $z\in B(F)$ is called a \textit{bifurcation value} of $F$. 
\end{definition}

The set $C(F)$ of critical values is contained in $B(F)$, but there may be bifurcation values that are not critical values:

\begin{example}
If $X = \mathbb{C}^2$ and $F = x + x^2 y$ then it is easy to see that $F$ has no critical values; in particular $\mathbb{H}^\ast (\mathrm{Crit}F , \mathscr{P}_{F}^\bullet ) = 0$. But $B(F) = \{ 0 \}$. Indeed, 
\[
F^{-1}(0) \cong \mathbb{C} \sqcup \mathbb{C}^\ast \quad , \quad F^{-1}(t\neq 0 ) \cong \mathbb{C}^\ast .
\]

\end{example}

\begin{definition}
We say $z \in \mathbb{C}$ is a \textit{bifurcation value at infinity of} $F$ if for every $\delta>0$ and every compact subset $K \subset X$, the mapping
\[
F: F^{-1}(B_\delta (z)) \setminus K \to B_{\delta} (z )
\]
is \textit{not} a $C^\infty$ locally trivial fibration.
\end{definition}

The following result, combined with Theorem \ref{theorem:main1}, establishes Corollary \ref{corollary:sheaf}:
\begin{proposition}\label{proposition:sheafcomputation}
Let $X$ be a non-singular algebraic variety and $F$ a regular function on $X$. Suppose that $F$ has no bifurcation values at infinity.
Then there exists a constant $\delta>0$ such that for any $t \in \mathbb{C}$ with $0 < | t | < \delta$ there is an isomorphism
\[
\mathbb{H}^\ast ( \mathrm{Crit}F , \mathscr{P}_{F}^\bullet ) \cong \bigoplus_{z \in C(F)} H^{\ast+1} (F^{-1}(B_{\delta} (z)  )  , F^{-1}(z+ t ) ).
\]
\end{proposition}
\begin{proof}

The sheaf $\mathscr{P}_{F}^\bullet$ is supported on the critical locus, which is contained in the union of the finitely-many critical fibers of $F$. We may thus replace $X$ by the open `tube' $T = F^{-1} (B_{\delta} (z))$ with $\delta>0$, and suppose that $z$ is the only critical value, which for simplicity is taken to be $z = 0$. 
Since $F$ is a fibration over $B_{\delta}(z ) \setminus z$, then for $0 < |t | < \delta$ we have a homotopy-equivalence over $T$:
\[
\begin{tikzcd}
X_t \arrow{rr}{\simeq} \arrow{rd}{\iota_t} &  & E \arrow{ld}{j \circ \pi }\\
 & T &  
\end{tikzcd}
\]
and hence by \cite[Proposition 2.4.6]{dimca} there is an isomorphism in $D_{c}^b (T)$:
\begin{align*}
 R (\iota_t)_\ast (\iota_t )^\ast \mathbb{Z} \cong R (j \circ \pi )_\ast (j \circ \pi )_{\ast } \mathbb{Z} .
\end{align*}

Since $F$ has no bifurcation values at infinity, the inclusion $\iota_0 : X_0 \to T$ defines a homotopy equivalence. To see this, fix a Riemannian metric on $T$ which is a product on a trivialisation at infinity; the flow of the gradient $- \nabla |F|^2$ induces a continuous deformation retraction of $T$ onto $X_0$ by flowing to infinite time.
Since $\iota_0 : X_0 \to T$ is a homotopy equivalence, then we have an isomorphism in $D_{c}^b (T)$ (again, by \cite[Proposition 2.5.6]{dimca}): for every $A^\bullet \in D_{c}^b (T)$
\[
R (\iota_0 )_\ast (\iota_0 )^\ast A^\bullet \cong A^\bullet .
\]
Thus, combining these two isomorphisms we obtain 
\begin{align*}
R(\iota_0 )_\ast \psi_F (\underline{\mathbb{Z}} ) \cong R (j \circ \pi )_\ast (j \circ \pi )_{\ast } \underline{\mathbb{Z}} \cong R (\iota_t)_\ast (\iota_t )^\ast \underline{\mathbb{Z}} .
\end{align*}

Passing to hypercohomology, using basic formal properties about hypercohomology and direct and inverse image functors, (see \cite[\S 2.3]{dimca}) gives:
\begin{align*}
\mathbb{H}^\ast ( X_0 , \psi_F \underline{\mathbb{Z}} ) & \cong \mathbb{H}^\ast (T , R(\iota_0 )_\ast \psi_F \underline{\mathbb{Z}} )\\
& \cong \mathbb{H}^\ast ( T , R (\iota_t)_\ast (\iota_t )^\ast \underline{\mathbb{Z}} ) \\
& \cong \mathbb{H}^\ast (X_t , (\iota_t )^\ast \underline{\mathbb{Z}}  = \underline{\mathbb{Z}}) \\
& = H^\ast (X_t , \mathbb{Z} ). 
\end{align*}

Passing to cohomology in the exact triangle (\ref{trianglepsiphi}) and inserting the previous isomorphism gives a long exact sequence
\[
 H^\ast (X_0 , \mathbb{Z} ) = H^\ast (T , \mathbb{Z} )\xrightarrow{\iota_{t}^\ast} H^\ast ( X_t , \mathbb{Z} ) \to \mathbb{H}^\ast (X_0 , \phi_F \underline{\mathbb{Z}} ) \xrightarrow{[+1]}
\]
from which we deduce:
\[
\mathbb{H}^\ast (X_0 , \phi_F \underline{\mathbb{Z}} )\cong H^{\ast +1}(T , X_t  ; \mathbb{Z} ).
\]
Since $\mathscr{P}_{F}^\bullet = \phi_{F} \underline{\mathbb{Z}}$ in this case, we have obtained the required isomorphism.
\end{proof}

\printbibliography

\end{document}